\documentclass{compositio}

\usepackage{amssymb}
\usepackage{amsmath}

\newcommand{\aut}[0]{\operatorname{Aut}}

\newcommand{\diff}[0]{\operatorname{Diff}}  
\newcommand{\bir}[0]{\operatorname{Bir}}
\newcommand{\Pic}[0]{\operatorname{Pic}}
\newcommand{\GL}[0]{\operatorname{GL}}
\newcommand{\Sym}[0]{\operatorname{Sym}}

\newcommand{\Ker}[0]{\operatorname{Ker}}

\newcommand{\SO}[0]{\operatorname{SO}}

\newcommand{\p}[0]{{\mathbb P}}
\newcommand{\z}[0]{{\mathbb Z}}
\newcommand{\n}[0]{{\mathbb N}}
\renewcommand{\c}[0]{{\mathbb C}}
\newcommand{\C}{{\mathbb C}}
\newcommand{\m}{{\mathfrak m}}
\renewcommand{\r}[0]{{\mathbb R}} 

\newcommand{\im}{{\mathbf i}}
\newcommand{\pr}{\mathrm{pr}}

\renewcommand{\epsilon}[0]{\varepsilon}

\newtheorem{thm}[equation]{Theorem}%
\newtheorem{prop}[equation]{Proposition}
\newtheorem{lem}[equation]{Lemma}
\newtheorem{lemdef}[equation]{Lemma and Definition}
\newtheorem{cor}[equation]{Corollary}

\theoremstyle{remark}
\newtheorem{rem}[equation]{Remark}

\newtheorem*{ack}{Acknowledgments}

\theoremstyle{definition}

\newtheorem{dfn}[equation]{Definition}

\newcommand{\doublecov}[0]{\zeta}

\begin{document}

\title[Real conic bundles and very transitive actions]{Geometrically rational real conic bundles\\ and very transitive actions}
\author{J\'er\'emy Blanc}
\email{Jeremy.Blanc@unibas.ch}
\address{
Mathematisches Institut,
Universit\"at Basel,
Rheinsprung 21,
CH-4051 Basel,
Schweiz}

\author{Fr\'ed\'eric Mangolte}
\email{mangolte@univ-savoie.fr}
\address{Laboratoire de Math\'ematiques,
Universit\'e de Savoie, 73376 Le Bourget du Lac Cedex, France.
Tel.: +33 (0)4 79 75 86 60, Fax: +33 (0)4 79 75 81 42}


\begin{abstract}
In this article we study the transitivity of the group of automorphisms of real algebraic surfaces. We characterize real algebraic surfaces with very transitive automorphism groups.  We give applications to the classification of real algebraic models of compact surfaces:  these applications yield new insight into the geometry of the real locus, proving several surprising facts on this geometry. 
This  geometry can be thought of as a half-way point between the biregular and birational geometries.
\end{abstract}

\maketitle

\section{Introduction}

The group of automorphisms of a complex algebraic variety is small: indeed, it is  finite in general. Moreover, the group of automorphisms is $3$-transitive only if the variety is $\p^1_\c$. On the other hand, it was recently proved that for a  surface $X(\r)$  birational to $\p^2_\r$, its group of automorphisms acts $n$-transitively on $X(\r)$ for any $n$. The main goal of this paper is to determine all real algebraic surfaces $X(\r)$ having a group of automorphisms which acts very transitively on $X(\r)$. For precise definitions and statements, see below.

The aim of this paper is to study the action of birational maps on the set of real points of a real algebraic variety. 
Let us emphasize a common terminological source of confusion about the meaning of what is a \emph{real algebraic variety} (see also the enlightening introduction of \cite{ko-topo-2000}). 
From the point of view of general algebraic geometry, a real variety $X$ is a variety defined over the real numbers, and a morphism is understood as being defined over all the geometric points. In most real algebraic geometry texts however, the algebraic structure considered corresponds to the algebraic structure of a neighbourhood of the real points $X(\r)$ in the whole complex variety -- or, in other words, the structure of a germ of an algebraic variety defined over $\r$.

From this point of view it is natural to view $X(\r)$ as a compact submanifold of~$\r^n$ defined by real polynomial equations, where~$n$ is some natural integer. 
Likewise, it is natural to say that a map $\psi \colon X(\r)\to Y(\r)$ is an \emph{isomorphism} if $\psi$ is induced by a birational map $\Psi \colon X\dasharrow Y$ such that $\Psi$ (respectively $\Psi^{-1}$) is regular at any point of $X(\r)$ (respectively of $Y(\r)$). 
In particular, $\psi \colon X(\r)\to Y(\r)$ is a diffeomorphism.
This notion corresponds to the notion of biregular maps defined in \cite[3.2.6]{bcr} for the structure of real algebraic variety commonly used in the context of real algebraic geometry.
To distinguish between the Zariski topology and the topology induced by the embedding of $X(\r) $ as a topological submanifold of $\r^n$, we will call the latter the \emph{Euclidean topology}. Throughout what follows, topological notions like connectedness or compactness will always refer to the Euclidean topology. 

Recall that a real projective surface is rational if it is birationally equivalent to the real projective plane, and that it is geometrically  rational if its complexification is birationally equivalent to the complex projective plane.
The number of connected components is a birational invariant. In particular, if $X$ is a rational projective surface, $X(\r)$  is connected. 

The paper \cite{hm3} proves that the group of automorphisms $\aut\bigl(X(\r)\bigr)$ acts $n$-transitively on $X(\r)$ for any $n$ and any rational real algebraic surface $X$. To study the case where $X(\r)$ is not connected, we have to refine the notion of $n$-transitivity. Indeed, if $X(\r)$ has non-homeomorphic connected components, then even the  group of self-homeomorphisms does not acts $2$-transitively.

\setcounter{equation}{-1}
\begin{dfn}
Let $G$ be a topological group acting continuously on a topological space $M$.
 We say that two $n$-tuples of distinct points $(p_1,\dots,p_n)$ and $(q_1,\dots,q_n)$ are \emph{compatible} if there exists an homeomorphism $\psi \colon M \to M$ such that $\psi(p_i)=q_i$ for each $i$. 
 The action of $G$ on $M$ is then said to be \emph{very transitive} if for any pair of compatible $n$-tuples of points $(p_1,\dots,p_n)$ and $(q_1,\dots,q_n)$ of $M$, there exists an element $g \in G$ such that $g(p_i)=q_i$ for each $i$. More generally, the action of $G$ is said to be very transitive \emph{on each connected component} if we require the above condition only in case, for each $i$, $p_i$ and $q_i$ belong to the same connected component of $M$.  
\end{dfn}

Up till now, it was not known when the automorphism group of a real algebraic surface is big. We give a complete answer to this question: this is one of the main result of this paper. Let $\#M$ be the number of connected components of a compact manifold $M$.

\begin{thm}\label{Thm:3compTrans}
Let $X$ be a nonsingular real projective surface. The group $\aut\bigl(X(\r)\bigr)$ is then very transitive on each connected component if and only if $X$ is geometrically rational and $\#X(\r) \leq3$.
 \end{thm}
 
In the three component case, Theorem~\ref{thm:verytrans} below says that the very transitivity of $\aut(X(\r))$ can be determined by examining the set of possible permutations of connected components. 

\begin{thm}\label{thm:verytrans}
Let $X$ be a nonsingular real projective surface. The group $\aut\bigl(X(\r)\bigr)$ then has a very transitive action on $X(\r)$ if and only if the following hold:
\begin{enumerate}
\item\label{thm:verytrans.1}  $X$ is geometrically rational, and
\item\label{thm:verytrans.2} 
\begin{enumerate}
\item $\#X(\r)\leq 2$, or 
\item $\#X(\r)=3$, and there is no pair of homeomorphic connected components, or
\item $\#X(\r)=M_1\sqcup M_2 \sqcup M_3$, $M_1\sim M_2 \not\sim M_3$, and there is a morphism $\pi \colon X \to \p^1_{\r}$ whose general fibres are rational curves, and an automorphism of~$\p^1_{\r}$ which fixes $\pi(M_3)$ and exchanges $\pi(M_1),\pi(M_2)$, or
\item $\#X(\r)=M_1\sqcup M_2 \sqcup M_3$, $M_1\sim M_2 \sim M_3$, and there is a morphism $\pi \colon X \to \p^1_{\r}$ whose general fibres are rational curves, such that any permutation of the set of intervals  $\bigl\{\pi(M_1),\pi(M_2),\pi(M_3)\bigr\}$ 
 is realised by an automorphism of~$\p^1_{\r}$.
 \end{enumerate}
\end{enumerate}

Furthermore, when $\aut\bigl(X(\r)\bigr)$ is not very transitive, it is not even  $2$-transitive.
\end{thm}

This theorem will be proved in Section~\ref{Sec:Proofs}.
Note that when $\#X(\r)> 3$, either any element of $\aut\bigl(X(\r)\bigr)$ preserves a conic bundle structure (Theorem~ \ref{Thm:Birapport}), or $\aut\bigl(X(\r)\bigr)$ is countable  (Corollary~\ref{Cor:DPcountable}): in either case $\aut\bigl(X(\r)\bigr)$ is not $1$-transitive.

These two theorems apply to the classification of algebraic models of real surfaces.
Up to this point in the paper $X(\r)$ is considered as a submanifold of some $\r^n$. Conversely, let $M$ be a compact $\mathcal{C}^\infty$-manifold. 
By the Nash-Tognoli theorem \cite{to}, every such $M$ is diffeomorphic to a nonsingular real algebraic subset of $\r^m$ for some $m$. 
Taking the Zariski closure in $\p^m$ and applying Hironaka's resolution of singularities \cite{hiro}, it follows that $M$ is in fact diffeomorphic to the set of real points $X(\r)$ of a nonsingular projective algebraic variety $X$ defined over $\r$. 
Such a variety $X$ is called an \emph{algebraic model} of $M$. A natural question is to classify the algebraic models of $M$ up to isomorphism for a given manifold $M$.

There are several recent results about algebraic models and their automorphism groups \cite{bh,hm3,hm4,km1}.
For example, when $M$ is $2$-dimensional, and admits a real rational algebraic model, this rational algebraic model is unique \cite{bh}. 
In other words, if $X$ and $Y$ are two rational real algebraic surfaces, then $X(\r)$ and $Y(\r)$ are isomorphic if and only if there are homeomorphic. We extend the classification of real algebraic models to geometrically rational surfaces.

\begin{thm}\label{thm:unicity}
Let $X,Y$ be two nonsingular geometrically rational real projective surfaces, and assume that $\#
X(\r)\leq 2$. The surface  $X(\r)$ is then isomorphic to $Y(\r)$ if and only
if $X$ is birational to $Y$ and $X(\r)$ is homeomorphic to $Y(\r)$. This is false in general when $\#X(\r)\geq 3$.
\end{thm}

Recall that a nonsingular projective surface is minimal if any birational morphism to a nonsingular surface is an isomorphism.
We have the following rigidity result on minimal geometrically rational real surfaces.

\begin{thm}\label{Thm:Diff+Bir=DiffBir}
    Let  $X$ and $Y$ be two minimal geometrically rational real projective surfaces, and assume that either $X$ or $Y$ is non-rational. The following are then equivalent:
\begin{enumerate}
\item
$X$ and $Y$ are birational.
 \item
 $X(\r)$ and $Y(\r)$  are isomorphic.
  \end{enumerate}
  \end{thm}

In this work, we classify the birational classes of real conic bundles and correct an error contained in the literature (Theorem~\ref{Thm:Birapport}). 
It follows that the only geometrically rational surfaces $X(\r)$ for which equivalence by homeomorphism implies equivalence by isomorphism are the connected ones. In particular, this yields a converse to \cite[Corollary~8.1]{bh}.
\begin{cor}\label{Cor:Models}
Let $M$ be a compact $\mathcal{C}^\infty$-surface. The surface $M$ then admits a unique geometrically rational model if and only if the following two conditions hold:
\begin{enumerate}
\item $M$ is connected, and
\item  $M$ is non-orientable or $M$ is orientable with genus $g(M) \leq 1$.
\end{enumerate}
\end{cor}
For $M$ orientable with $g(M) > 1$, no uniqueness result -- even very weak --holds. We can therefore ask what the simplest algebraic model for such an $M$ should be. This question is studied in the forthcoming paper \cite{hm5}.

\medskip

Another way of measuring the size of $\aut(X(\r))$ was used in \cite{km1}, where it is proved that for any rational surface $X$, $\aut\bigl(X(\r)\bigr)\subset \diff\bigl(X(\r)\bigr)$ is dense for the strong topology. 
 For non geometrically rational surfaces and for most of the non-rational geometrically rational surfaces, the group  $\aut\bigl(X(\r)\bigr)$ cannot be dense. 
 The above paper left the question of density open only for certain geometrically rational surfaces with $2$, $3$, $4$ or $5$ connected components. 
 One by-product of our results is the non-density of  $\aut(X(\r))$ for most surfaces with at least $3$ connected components -- see Proposition~\ref{Prp:Density}.
 
 Let us mention some  other papers on automorphisms of real projective surfaces.
In \cite{rv}, it is proved that $\aut\bigl(\p^2(\r)\bigr)$ is generated by linear automorphisms
 and certain real algebraic automorphisms of degree 5. 
The paper \cite{hm4} is devoted to the study of very transitive actions and uniqueness of models for some kind of singular rational surfaces.

\subsection*{Strategy of the proof}
In the proof of Theorem~\ref{Thm:3compTrans}, the main part concerns minimal conic bundles. We first prove that two minimal conic bundles are isomorphic if they induce the same intervals on the basis. Given a set of intervals, one choose the most special conic bundle, the so-called exceptional conic bundle, to write explicitly the automorphisms and to obtain a fiberwise transitivity. We then use the most general conic bundles which come with distinct foliations on the same surface. The foliations being transversal, this yields the very transivity of the automorphism group in the minimal case. 
\subsection*{Outline of the article}
In Section~\ref{Sec:Nota} we fix notations and  in Section~\ref{Sec:Min}  we recall the classification of minimal geometrically rational real surfaces. 

Section~\ref{Sec:conicbundle}, which constitutes the technical heart of the paper, is devoted to conic bundles, especially  minimal ones. We give representative elements of isomorphism classes, and explain the links between the various conic bundles.

In Section~\ref{Sec:ConicBunDPS}, we investigate real surfaces which admit two conic bundle structures. In particular, we  show that these are del Pezzo surfaces, and give descriptions of the possible conic bundles on these surfaces.
Section~\ref{Sec:Equiv} is devoted to the proof of Theorem~\ref{Thm:Diff+Bir=DiffBir}. We firstly correct an inaccuracy in the literature, by proving that if two surfaces admitting a conic bundle structure are birational, then the birational map may be chosen so that it preserves the conic bundle structures. We then strengthen this result to isomorphisms between real parts when the surfaces are minimal, before proving Theorem~\ref{Thm:Diff+Bir=DiffBir}.

In Section~\ref{Sec:VeryTransitive}, we prove that if the real part  of a minimal geometrically rational surface has $2$ or $3$ connected components, then its automorphism group is very transitive on each connected component. In Section~\ref{Sec:RealAlgModels}, we prove the same result for non-minimal surfaces. We show how to separate infinitely close points, which is certainly one of the most counter-intuitive aspects of our geometry, and was first observed in \cite{bh} for rational surfaces. We also prove the uniqueness of models in many cases. 

In Section~\ref{Sec:Proofs}, we then use all the results of the preceding sections to prove the main results stated in the introduction (except Theorem~\ref{Thm:Diff+Bir=DiffBir}, which is proved in Section~\ref{Sec:Equiv}).

\begin{ack}
We are grateful to the referee
for helpful remarks that enabled us to shorten several proofs and improve the presentation of this article.
\end{ack}

\section{Notation}\label{Sec:Nota}

Throughout what follows, by a variety we will mean an algebraic variety, which may be real or complex (i.e.\ defined over $\r$ or $\c$). 
If the converse is not expressly stated all our varieties will be projective and all our surfaces will be nonsingular and geometrically rational (i.e.\ rational over $\c$).

Recall that a real variety $X$ may be identified with a pair $(S,\sigma)$, where $S$ is a complex variety and $\sigma$ is an anti-holomorphic involution on $S$; by abuse of notation we will write $X=(S,\sigma)$. 
Then, $S(\c)=X(\c)$ denotes the set of complex points of the variety, and $X(\r)=S(\c)^{\sigma}$ is the set of real points. 
A point $p\in X$ may be real (if it belongs to $X(\r)$), or imaginary (if it belongs to $X(\c)\backslash X(\r)$). If $X(\r)$ is non empty (which will be the case for all our surfaces), then $\Pic(X)\cong\Pic(S)^{\sigma}$, \cite[I.(4.5)]{Sil}. 
As we only work with regular surfaces (i.e. $q(X)=q(S)=0$), the Picard group is isomorphic to the N\'eron-Severi group, and $\rho(S)$ and $\rho(X)$ will denote the rank of $\Pic(S)$ and $\Pic(X)$ respectively.  
Recall that $\rho(X)\leq \rho(S)$.  We denote by $K_X\in\Pic(X)$ the canonical class, which may be identified with $K_S$.  
The intersection of two divisors of $\Pic(S)$ or $\Pic(X)$ will always denote the usual intersection in $\Pic(S)$.

We will use the classical notions of morphisms, rational maps, isomorphisms and automorphisms between real or complex varieties. 
Moreover, if $X_1$ and $X_2$ are two real varieties, an isomorphism \emph{between real parts} $X_1(\r)\stackrel{\psi}{\to} X_2(\r)$ is a birational map $\psi\colon X_1\dasharrow X_2$ such that $\psi$ (respectively $\psi^{-1}$) is regular at any point of $X_1(\r)$ (respectively of $X_2(\r)$). 
This endows $X_{1}(\r)$ with a structure of a germ of algebraic variety defined over $\r$ (as in \cite[3.2.6]{bcr}), whereas the structure of $X_{1}$ is the one of an algebraic variety.

This notion of isomorphism between real parts gives rise to a geometry with rather unexpected properties comparing to those of the biregular geometry or the birational geometry. For example, let $\alpha \colon X_1(\r) \to X_2(\r)$ be an isomorphism, and $\varepsilon \colon Y_1 \dasharrow X_1$, $\eta \colon Y_2 \dasharrow X_2$ be two birational maps; the map $\psi := \varepsilon^{-1}\alpha\eta$ is a well-defined  birational map. Then $\psi$ can be an isomorphism $Y_1(\r) \to Y_2(\r)$ even if neither $\varepsilon$, nor $\eta$ is an isomorphism between real parts. In the same vein, let $\alpha \colon X_1(\r) \to X_2(\r)$ be an isomorphism, and let $\eta_{1}\colon Y_{1}\to X_{1}$ and $\eta_{2}\colon Y_{2}\to X_{2}$ be two birational morphisms which are the blow-ups of only real points (which may be proper or infinitely near points of $X_{1}$ and $X_{2}$). If $\alpha$ sends the points blown-up by $\eta_{1}$ on the points blown-up by $\eta_{2}$, then $\beta=(\eta_{2})^{-1}\alpha\eta_{1}\colon Y_{1}(\r)\to Y_{2}(\r)$ is an isomorphism.

Using $\aut$ and $\bir$ to denote respectively the group of automorphisms and birational self-maps of a variety, we have the following inclusions for the groups associated to $X=(S,\sigma)$:
\[\begin{array}{ccccc}
\aut(S) && \subset && \bir(S)\\
\cup && && \cup\\
\aut(X)&\subset& \aut\bigl(X(\r)\bigr)&\subset &\bir(X)\;.\end{array}\]

By $\p^n$ we mean the projective $n$-space, which may be complex or real depending on the context. It is unique as a complex variety -- written $\p^n_{\c}$. 
However, as a real variety, $\p^n$ may either be  $\p^n_{\c}$ endowed with the standard anti-holomorphic involution, written $\p^n_{\r}$, or only when  $n$ is odd, $\p^n_{\c}$ with a special involution with no real points, written $(\p^n,\emptyset)$. 
To lighten notation, and since we never speak about $(\p^1,\emptyset)(\r)$ we write $\p^1(\r)$ for $\p^1_{\r}(\r)$. 

\section{Minimal surfaces and minimal conic bundles}\label{Sec:Min}

The aim of this section is to reduce our study of geometrically rational surfaces to surfaces which admit a minimal conic bundle structure. We first recall the classification of geometrically rational surfaces (see \cite{Sil} for an introduction).
The proofs of Theorems~\ref{thm:verytrans} and~\ref{Thm:Diff+Bir=DiffBir} will then split into three cases: rational, del Pezzo with $\rho=1$, and minimal conic bundle. 
The rational case is treated in \cite{hm3} and
Proposition~\ref{Prop:DP1DP2rho1} below states the case of a del Pezzo surface with $\rho=1$.
 


  \begin{dfn}\label{dfn.birat.conic}
  A \emph{conic bundle} is a pair $(X,\pi)$ where $X$ is a surface and $\pi$ is a morphism $X\rightarrow \p^1$, where any fibre of $\pi$ is isomorphic to a plane conic.   If $(X,\pi)$ and  $(X',\pi')$ are two conic bundles, a \emph{birational map of conic bundles} $\psi\colon (X,\pi)\dasharrow (X',\pi')$ is a birational map $\psi\colon X\dasharrow X'$ such that there exists an automorphism $\alpha$ of $\p^1$ with $\pi'\circ\phi=\pi\circ\alpha$.
  \end{dfn}
  

We will assume throughout what follows that if $X$ is real, then the basis is $\p^1_{\r}$ (and not $(\p^1,\emptyset)$). This avoids certain conic bundles with no real points.
  We denote by $\aut(X,\pi)$ (respectively $\bir(X,\pi)$) the group of automorphisms (respectively birational self-maps) of the conic bundle $(X,\pi)$. Observe that $\aut(X,\pi)=\aut(X)\cap \bir(X,\pi)$. Similarly, when $(X,\pi)$ is real we denote by $\aut(X(\r),\pi)$ the group $\aut\bigl(X(\r)\bigr)\cap \bir(X,\pi)$.

  Recall that a real algebraic  surface $X$ is \emph{minimal} if and only if there is no real $(-1)$-curve and no pair of disjoint conjugate imaginary  $(-1)$-curves on $X$, and that a real conic bundle $(X,\pi)$ is minimal if and only if the two irreducible components of any real singular fibre of $\pi$ are imaginary. Compare to the complex case where $(X,\pi)$ is minimal if and only if there is no singular fibre.



\bigskip


The following two results follow from the work of Comessatti \cite{Com1},  (see also \cite{bib:Man}, \cite{bib:IskMinimal}, \cite[Chap.~V]{Sil}, or \cite{Kol}).
Recall that a surface $X$ is a del Pezzo surface if the anti-canonical divisor $-K_X$ is ample. The same definition applies for $X$ real or complex. 

  \begin{thm}\label{Thm:ClassicMinimal}
  If $X$ is a minimal geometrically rational real surface such that $X(\r)\ne \emptyset$, then one and exactly one of the following holds:
  \begin{enumerate}
  \item
  $X$ is rational: it is isomorphic to $\p^2_{\r}$, to the quadric $Q_0:=\{(x:y:z:t)\in\p^3_\r\ |\ x^2+y^2+z^2=t^2\}$, or to a real Hirzebruch surface $\mathbb{F}_n$, $n\ne 1$;
  \item
  $X$ is a del Pezzo surface of degree $1$ or $2$ with $\rho(X)=1$;
  \item
  there exists a minimal conic bundle structure $\pi\colon X\rightarrow \p^1$ with an even number of singular fibres $2r\geq 4$. Moreover, $\rho(X)=2$.
  \end{enumerate}
  \end{thm}
  
  \begin{rem}
 If $(S,\sigma)$ is a minimal geometrically rational real surface such that $S(\c)^\sigma= \emptyset$, then $S$ is an Hirzebruch surface of even index.
  \end{rem}

\begin{prop}[Topology of the real part] \label{Prp:TopMinimal}
In each case of the former theorem, we have:
\begin{enumerate}
\item $X$ is rational if and only if $X(\r)$ is connected. 
When $X$ is moreover minimal, then $X(\r)$ is homeomorphic to one of the following: the real projective plane, the sphere, the torus, or the Klein bottle.
\item When $X$ is a minimal del Pezzo surface of degree~$1$, it satisfies $\rho(X)=1$, and $X(\r)$ is the disjoint union of one real projective plane and $4$ spheres. 
If $X$ is a minimal del Pezzo surface of degree~$2$ with $\rho(X)=1$, then $X(\r)$ is the disjoint union of $4$ spheres.
\item If $X$ is non-rational and is endowed with a minimal conic bundle with $2r$ singular fibres, then $X(\r)$ is the disjoint union of $r$ spheres, $r \geq 2$.
\end{enumerate}
\end{prop}

\begin{prop}\label{Prop:DP1DP2rho1}
Let $X,Y$ be two minimal geometrically rational real surfaces. Assume that $X$ is not rational and satisfies $\rho(X)=1$ $($but $\rho(Y)$ may be equal to $1$ or $2)$.

\begin{enumerate}
\item
If $X$ is a del Pezzo surface of degree~$1$, then any birational map $X\dasharrow Y$ is an isomorphism. 
In particular, 
$$
\aut(X)=\aut\bigl(X(\r)\bigr)=\bir(X)\;.
$$

\item
If $X$ is a del Pezzo surface of degree~$2$, $X$ is birational to $Y$ if and only $X$ is isomorphic to $Y$. 
Moreover, all the base-points of the elements of $\bir(X)$ are real, and
$$
\aut(X)=\aut\bigl(X(\r)\bigr)\subsetneq\bir(X)\;.
$$

\end{enumerate}
\end{prop}

\begin{proof}
Assume the existence of a birational map $\psi\colon X\dasharrow Y$. If $\psi$ is not an isomorphism, we decompose $\psi$ into elementary links \[X=X_{0}\stackrel{\psi_{1}}{\dasharrow}X_{1}\stackrel{\psi_{2}}{\dasharrow}\dots \stackrel{\psi_{n-1}}{\dasharrow}X_{n-1}\stackrel{\psi_{n}}{\dasharrow}X_{n}=Y\]
as in \cite[Theorem 2.5]{IskFact}. It follows from the description of the links of \cite[Theorem 2.6]{IskFact} that for any  link $\psi_{i}\colon X_{i-1}\dasharrow X_{i}$, $X_{i-1}$ and $X_{i}$ are isomorphic del Pezzo surfaces of degree $2$, and that $\psi_{i}$ is equal to $\beta\eta\alpha\eta^{-1}$, where $\eta$ is the blow-up $X'\to X_{i-1}$ of a real point of $X_{i-1}$, $X'$ is a del Pezzo surface of degree $1$,  $\alpha\in \aut(X')$ is the Bertini involution of the surface, and $\beta\colon X_{i+1}\to X_{i}$ is an isomorphism.

Therefore, $Y$ is isomorphic to $X$. Moreover, if $X$ has degree~$1$, $\psi$ is an isomorphism. If $X$ has degree $2$, $\psi$ is decomposed into conjugates of Bertini involutions, so each of its base-points is real. This proves that if $\psi\in \aut\bigl(X(\r)\bigr)$ then $\psi\in \aut(X)$. 
Furthermore, conjugates of Bertini involutions belong to $\bir(X)$ but not to $\aut(X)=\aut\bigl(X(\r)\bigr)$.
\end{proof}
\begin{cor}\label{Cor:DPcountable}
Let $X_{0}$ be a minimal non-rational geometrically rational real surface with $\rho(X_{0})=1$, and let $\eta\colon X\to X_{0}$ be a birational morphism.

Then, $\aut\bigl(X(\r)\bigr)$ is countable. 
Moreover, if $X_{0}$ is a del Pezzo surface of degree~$1$, then $\aut\bigl(X(\r)\bigr)$ is finite.
\end{cor}

\begin{proof}
Without changing the isomorphism class of $X(\r)$ we may assume that $\eta$ is the blow-up of only real points (which may belong to $X_{0}$ as proper or infinitely near points). Since any base-point of any element of $\bir(X_{0})$ is real (Proposition~\ref{Prop:DP1DP2rho1}), the same is true for any element of $\bir(X)$. In particular, $\aut\bigl(X(\r)\bigr)=\aut(X)$. 
The group $\aut(X)$ acts on $\Pic(X)\cong \mathbb{Z}^{n}$, where $n=\rho(X)\geq 1$. 
This action gives rise to an homomorphism $\theta\colon \aut(X)\to \GL(n,\z)$. 
Let us prove that $\theta$ is injective. Indeed, if $\alpha\in \Ker(\theta)$, then $\alpha$ is conjugate by $\eta$ to an element  of $\alpha_{0}\in \aut(X_{0})$ which acts trivially on $\Pic(X_{0})$. 
Writing $S_{0}$ the complex surface obtaining by forgetting the real structure of $X_{0}$, $S_{0}$ is the blow-up of $7$ or $8$ points in general position of $\p^2_{\c}$. 
Thus $\alpha_{0}\in \aut(X_{0})\subset \aut(S_{0})$ is the lift of an automorphism of $\p^2_{\c}$ which fixes $7$ or $8$ points, no $3$ collinear, hence is the identity.

The morphism $\theta$ is injective, and this shows that $\aut\bigl(X(\r)\bigr)=\aut(X)$ is countable. Moreover, if $X_{0}$ is a del Pezzo surface of degree $1$, then $\bir(X_{0})=\aut(X_{0})$ (by Proposition~\ref{Prop:DP1DP2rho1}). Since $\aut(X_{0})$ is finite, $\aut\bigl(X(\r)\bigr)\subset \bir(X)$ is also finite.
\end{proof}
  
\section{Minimal and exceptional conic bundles}\label{Sec:conicbundle}

\begin{dfn}
If  $(X,\pi)$ is a real conic bundle, $I(X,\pi)\subset \p^1(\r)$ denotes the image by $\pi$ of the set $X(\r)$ of real points of $X$. 
\end{dfn} 
The set $I(X,\pi)$ is the union of a finite number of intervals (which may be $\emptyset$ or $\p^1(\r)$), and it is well-known that it determines the birational class of $(X,\pi)$. 
We prove that $I(X,\pi)$ also determines the equivalence class of $(X(\r),\pi)$ among the minimal conic bundles, and give the proof of Theorem~\ref{Thm:Diff+Bir=DiffBir} in the case of conic bundles (Corollary~\ref{Cor:Thm5Conic}). 

\bigskip

\begin{lemdef}\label{LemDef:Exc}
Let $(X,\pi)$ be a real minimal conic bundle. The following conditions are equivalent:

\begin{enumerate}
\item\label{LemDef:Exc1} There exists a section $s$ such that $s$ and $\bar{s}$ do not intersect.

\item\label{LemDef:Exc2} There exists a section $s$ such that $s^2=-r$, where $2r$ is the number of singular fibres.
\end{enumerate}

\noindent If any of these conditions occur, we say that $(X,\pi)$ is \emph{exceptional}.
\end{lemdef}
\begin{proof}
Let $s$ be a section satisfying one of the two conditions. Denote by $(S,\pi)$ the complex conic bundle obtained by forgetting the real structure of $(X,\pi)$, and by $\eta:X\to \mathbb{F}_m$ the birational map which contracts in any singular fibre of $\pi$ the irreducible component which does not intersect $s$. If $s$ satisfies condition $\ref{LemDef:Exc1})$, $\eta(\bar{s})$ and $\eta(s)$ are two sections of $\mathbb{F}_m$ which do not intersect, so they have self-intersections $-m$ and $m$. This means that $s^2=\bar{s}^2=-m$ and that the number of singular fibres is $2m$, and implies $\ref{LemDef:Exc2})$. Conversely, if $s$ satisfies $\ref{LemDef:Exc2})$, $\eta(s)$ and $\eta(\bar{s})$ are sections of $\mathbb{F}_m$ of self-intersection $-r$ and $r$. If these two sections are distinct, they do not intersect, which means that $s$ and $\bar{s}$ do not intersect. If $\eta(s)=\eta(\bar{s})$, we have $r=0$, and $X=(\p^1_\c\times\p^1_\c,\sigma)$ for a certain anti-holomorphic involution $\sigma$. We may thus choose another section $s'$ of self-intersection $0$ which is imaginary. 
\end{proof}
\begin{rem}
The definition of exceptional conic bundles was introduced  in \cite{DoIs} and \cite{BlaTG} for \emph{complex} conic bundles endowed with an \emph{holomorphic} involution. If $(S,\pi)$ is an exceptional complex conic bundle with at least $4$ singular fibres, $\aut(S,\pi)=\aut(S)$ is a maximal algebraic subgroup of $\bir(S)$ \cite{BlaTG}.
\end{rem}

\begin{lem}\label{Lem:IsoEx}
Let $(Y,\pi_Y)$ be a minimal real conic bundle such that $\pi_Y$ has at least one singular fibre. There exists an \emph{exceptional} real conic bundle $(X,\pi_X)$ and an isomorphism $\psi\colon Y(\r)\to X(\r)$ such that $\pi_X\circ\psi=\pi_Y$.
\end{lem}
\begin{rem}
The result is false without the assumption on the number of singular fibres. Consider for example $Y=\mathbb{F}_3(\r)$, whose real part is homeomorphic to the Klein bottle. Indeed, any exceptional conic bundle with no singular fibres is a real form of $(\p^1_\c\times\p^1_\c,\pr_1)$, and thus has a real part either empty or homeomorphic to the torus $S^1\times S^1$. 
\end{rem}
Before proving Lemma~\ref{Lem:IsoEx}, we associate  to any given exceptional conic bundle $X$ an explicit circle bundle isomorphic to it. The following improves \cite[Cor.VI.3.1]{Sil} where the model is only assumed birational to $X$.

\begin{lem}\label{Lem:AffEx}
Let $(X,\pi)$ be an exceptional real conic bundle. Then, there exists an affine real variety $A\subset X$ isomorphic to the affine surface of $\r^3$ given by
\[y^2+z^2=Q(x),\]
where $Q$ is a real polynomial with only simple roots, all real. Moreover, $\pi|_A\colon A\to \p^1_\r$ is the projection $(x,y,z)\mapsto (x:1)$, and $I(X,\pi)$ is the closure of $\{(x:1)\in\p^1_\r\ |\ Q(x)\geq 0\}$. 

Furthermore, if $f=\pi^{-1}((1:0))\subset X$ is a nonsingular fibre, the singular fibres of $\pi$ are those of the points $\{(x:1)\ |\ Q(x)=0\}$ and the inclusion $A\to X$ is an isomorphism $A(\r)\to \left(X\backslash f\right)(\r)$.
In particular, if $(1:0)\notin I(X,\pi)$, the inclusion yields an isomorphism $A(\r)\to X(\r)$.
\end{lem}
\begin{proof}Denote by $2r$ the number of singular fibres of $\pi$ (which is even, see Lemma~\ref{LemDef:Exc}). 

Assume first that $r=0$, which implies that $(X,\pi)$ is a real form of $(\p^1_\c\times\p^1_\c,\pr_1)$, hence is isomorphic to $(\p^1_\r\times \p^1_\r,\pr_1)$ or to $(\p^1_\r\times (\p^1,{\emptyset}),\pr_1)$, see convention after Definition~\ref{dfn.birat.conic}. Taking $Q(x)=1$ or $Q(x)=-1$ gives the result.

Assume now that $r>0$, and denote by $s$ and $\bar{s}$ two conjugate imaginary sections of $\pi$ of self-intersection $-r$. Changing $\pi$ by an automorphism of $\p^1$, we can assume that $(1:0)$.  The singular fibres of $\pi$ are above the points $(a_1:1),\dots,(a_{2r}:1)$, where the $a_i$ are distinct real numbers.  Let $J = (J_1, J_2)$ be a partition of $\{a_1,\dots,a_{2r}\}$ into two sets of $r$ points. Let $\eta$ be the birational morphism (not defined over $\r$) which contracts the irreducible component of $\pi^{-1}((a_i:1))$ which intersects $s$ if $a_i\in J_1$ and the component which intersects $\bar{s}$ if $a_i\in J_2$. Then, the images of $s$ and $\bar{s}$ are two sections of self-intersection $0$. Thus we may assume that $\eta$ is a birational morphism of conic bundles $(S,\pi)\rightarrow (\p^1_\c\times\p^1_\c,\pr_1)$, where $S$ is the complex surface obtained by forgetting the real structure of $X$, $\pr_1$ is the projection on the first factor, and where $\eta(s)$ and $\eta(\bar{s})$ are equal to $\p^1_\c\times(0:1)$ and $\p^1_\c\times(1:0)$.

We write $P_{1}(x_{1},x_{2})=\prod_{a \in J_1} ({x_1}-a {x_2})$ and $P_{2}(x_{1},x_{2})=\prod_{a \in J_2} ({x_1}-a {x_2})$, and denote by $\alpha$ and $\sigma$ the self-maps of $S$, which are the lifts by $\eta$ of the
  following self-maps of $\p^1_{\c}\times \p^1_{\c}$:
 \[\begin{array}{l}\alpha'\colon\big( (x_1:x_2),(y_1:y_2)\big)\dasharrow \big(({x_1}:{x_2}),(-{y_2}\cdot P_{1}({x_1},{x_2}):{y_1}\cdot P_{2}(x_{1},x_{2})\big),\\ \vphantom{\Big (}
 \sigma'\colon\big( (x_1:x_2),(y_1:y_2)\big)\dasharrow \big((\overline{x_1}:\overline{x_2}),(-\overline{y_2}\cdot P_{1}(\overline{x_{1}},\overline{x_{2}}):\overline{y_1}\cdot P_{2}(\overline{x_{1}},\overline{x_{2}})\big).\end{array}\]

The map $\alpha'$ is a birational involution of $\p^1_{\c}\times \p^1_{\c}$, which is defined over $\r$, and whose base-points are precisely the $2r$ points $\{((x:1),(0:1))\ |\ x\in J_1\} \cup \{((x:1),(1:0))\ |\ x\in J_2\}$ blown-up by $\eta$. Since $\alpha'$ is an involution and $\eta$ is the blow-up of all of its base-points, $\alpha=\eta^{-1}\alpha'\eta$ is an automorphism of $S$, which belongs to $\aut(S,\pi)$. In consequence, $\sigma$ is an anti-holomorphic involution of $S$. 

Denote by $\sigma_X$ the anti-holomorphic involution on $S$ which gives the real structure of $X$. 
The map $\sigma_X\circ\sigma^{-1}$ belongs to $\aut(S,\pi)$ and acts trivially on the basis, since $\sigma$ and $\sigma_X$ have the same action on the basis. Moreover, since both $\sigma_X$ and $\sigma$ exchange the irreducible components of each singular fibre, $\sigma_X\circ\sigma^{-1}$ preserves any curve contracted by $\eta$ and is therefore the lift by $\eta$ of $\beta\colon \big((x_1:x_2),(y_1:y_2)\big)\mapsto \big((x_1:x_2),(\mu y_1:y_2)\big)$ for some $\mu\in \c^{*}$. It follows that $\sigma_X'=\eta\circ\sigma_X\circ\eta^{-1}=\beta\circ \sigma'$ is the map 
\[\sigma_X'\colon\big((x_1:x_2),(y_1:y_2)\big)\dasharrow \big((\overline{x_1}:\overline{x_2}),(-{\mu\cdot}\overline{y_2} P_1(\overline{x_1},\overline{x_2}):\overline{y_1} P_2(\overline{x_1},\overline{x_2}))\big).\]

Let us write $Q(x)=-\mu P_1(x,1)P_2(x,1)$, denote by $B\subset \C^3$ the affine hypersurface of equation $y^2+z^2=Q(x)$, and by $\pi_B\colon B\rightarrow \p^1$ the map $(x,y,z)\mapsto (x:1)$. Let $A=(B,\sigma_B)$, where $\sigma_B$ sends $(x,y,z)$ onto $(\bar{x},\bar{y},\bar{z})$.
Denote by $\theta\colon B\dasharrow \p^1_\c\times\p^1_\c$ the map that sends $(x,y,z)$ onto  $\big((x:1),(y-\im z:P_2(x,1))\big)$ if $P_2(x,1)\not=0$ and onto $\big((x:1),(-\mu P_1(x,1):y+\im z)\big)$ if $P_1(x,1)\not=0$. Then $\theta$ is a birational morphism, and $\theta^{-1}$ sends $  \big((x_1:x_2),(y_1:y_2)\big)$ on
\[\left(\frac{x_1}{x_2},\frac{1}{2}\left(\frac{y_1}{y_2}P_2(x_1,x_2)-\frac{y_2}{y_1} \mu P_1(x_1,x_2)\right),\frac{\im}{2}\left(\frac{y_1}{y_2}P_2(x_1,x_2)+\frac{y_2}{y_1} \mu P_1(x_1,x_2)\right)\right).\]

Observe that $\sigma_X' \theta=\sigma_B\theta$.
In consequence, $\psi=\eta^{-1}\circ\theta$ is a real birational map $A\dasharrow X$. 

Moreover, $\psi$ is an isomorphism from $B$ to the complement in $S$ of the union of $\pi^{-1}((1:0))$ and the pull-back by $\eta$ of $\p^1\times (0:1)$ and $\p^1\times (1:0)$. Indeed
let $x_0\in \C$. If $x_0\in \C$ is such that $Q(x_0)\not=0$, then $\theta$ restricts to an isomorphism from $\pi_B^{-1}((x_0:1))$ to  $\{((x_0:1),(y_1:y_2))\in \p^1_\c\times\p^1_\c\ |\ y_1y_2\not=0\}\cong \c^{*}$. If $Q(x_0)=0$, then $x_0\in J_1\cup J_2$, and the fibre $\pi_B^{-1}((x_0:1))$ consists of two lines of $\C^2$ which intersect, given by $y=\im z$ and $y=-\im z$. If $x_0\in J_1$, then the line $y+\im z=0$ is sent isomorphically by $\theta$ onto the fibre $\{((x_0:1),(y_1:y_2))\in \p^1_\c\times\p^1_\c\ |\ y_2\not=0\}\cong \c^{*}$, and the line $y-\im z$ is contracted on the point $((x_0:1),(0:1))$. The map $\psi$ sends thus isomorphically $\pi_B^{-1}((x_0:1))$ onto the fibre $\pi^{-1}((x_0:1))$ minus the two points corresponding to the two sections of self-intersection $-r$. The situation when $x_0\in J_2$ is similar. 

The map $\psi$ is therefore an inclusion $A\to X$ and, by construction, it satisfies all the properties stated in the lemma.
\end{proof}

\begin{proof}[Proof of Lemma~\ref{Lem:IsoEx}]Take a section $s$ of $\pi_Y$. If $s$ intersects its conjugate $\bar{s}$ into a real point $p$ (respectively into a pair of imaginary points $q_1$ and $q_2$), then blow-up the point $p$ (respectively $q_1$ and $q_2$), and contract the strict transform of the fibre of the blown-up point(s). Repeating this process, we obtain a minimal real conic bundle $(Z,\pi_Z)$ and a birational map $\phi\colon Y\dasharrow Z$ such that $\pi_Z\circ \phi= \pi_Y$ and $\phi(s)$ does not intersect its conjugate. 

If all the base-points of $\phi$ are imaginary, we set $\psi=\phi$ and $(X,\pi_X)=(Z,\pi_Z)$. Otherwise, by induction on the number of real base-points of $\phi$, it suffices to prove the existence of $\psi$ when $\phi$ is an elementary link centered at only one real point.

Denote by $q\in Z$ the real point which is the base-point of $\phi^{-1}$. Since $\pi_Y$ has at least one singular fibre, this is also the case for $\pi_Z$, and thus $I(Z,\pi_Z)$ is not the whole $\p^1(\r)$ (By Lemma~\ref{Lem:AffEx}). We may thus assume that $(1:0)\notin I(Z,\pi_Z)$, that $\pi_Z(q)=(1:1)$, and that the interval of $I(Z,\pi_Z)$ which contains $\pi_Z(q)$ is $\{(x:1)\in \mathbb{P}^1_\r\ |\ 0\leq x \leq a\}$ for some $a\in \r$, $a>1$. Take the affine surface $A\subset Z$ given by Lemma~\ref{Lem:AffEx}, which is  isomorphic to $y^2+z^2=Q(x)$ for some polynomial $Q$. Then, $Q(0)=Q(a)=0$ and $Q(x)>0$ for $0<x<a$, and we may assume that $Q(1)=1$. Denote by $s$ the section of $\pi_Z\colon Z\to \p^1_{\r}$ given locally by $y+\im z=\im x^n$, for some positive integer $n$. Its conjugate is given by $y-\im z=-\im x^n$, or $y+\im z=Q(x)/(-\im x^n)$. Thus, $s$ intersects $\bar{s}$ at some real point $p\in Z$, its image $x=\pi_Z(p)$ satisfies $Q(x)/(-\im x^n)=\im x^n$, or $Q(x)=x^{2n}$. Taking $n$ large enough, this can only happen when $x=0$ or $x=1$. The first possibility cannot occur since a section does not pass through the singular point of a singular fibre. Thus, $s$ intersects $\bar{s}$ at only one real point, which is $q$. In consequence, the strict pull-back by $\phi$ of $s$ is a section of $Y$ which intersects its conjugate at only imaginary points. This shows that $(Y(\r),\pi_Y)$ is isomorphic to an exceptional real conic bundle $(X,\pi_X)$.
 \end{proof}



\begin{cor}\label{Prp:Equivalences}
Let $(X,\pi_{X})$ and $(Y,\pi_{Y})$ be two minimal real conic bundles, and assume that either $\pi_X$ or $\pi_Y$ has at least one singular fibre. Then, the following are equivalent:

\begin{enumerate}
\item \label{Prp:Equivalences1} $I(X,\pi_X)=I(Y,\pi_Y)$;
\item \label{Prp:Equivalences3} there exists an isomorphism  $\varphi\colon X(\r)\to Y(\r)$ such that $\pi_Y\circ\varphi=\pi_X$.\end{enumerate}
\end{cor}
\begin{proof}

It suffices to prove $\ref{Prp:Equivalences1})\Rightarrow \ref{Prp:Equivalences3})$.  
By Lemma~\ref{Lem:IsoEx}, we may assume that both $(X,\pi_X)$ and $(Y,\pi_Y)$ are exceptional.
We may now assume that the fibre over $(1:0)$ is not singular and use Lemma~\ref{Lem:AffEx}: let $A_X\subset X$ and $B_X\subset Y$ be the affine surfaces given by the lemma, with equations $y^2+z^2=Q_X(x)$ and $y^2+z^2=Q_Y(x)$ respectively. Since $I(X,\pi_X)=I(Y,\pi_Y)$, $Q_Y(x)=\lambda Q_X(x)$ for some positive $\lambda\in \r$. The map $(x,y,z)\mapsto ( x,\sqrt{\lambda}y,\sqrt{\lambda} z)$ then yields an isomorphism $(X(\r),\pi_X)\to (Y(\r),\pi_Y)$.
\end{proof}
The above result implies the next two corollaries.  The first one strengthen a result of Comessatti  \cite{Com1} (see also  \cite[Theorem 4.5]{Kol}). 

\begin{cor}\label{Cor:BirAutI}
Let $(X,\pi)$ and $(X',\pi')$ be two real conic bundles.
Assume that $(X,\pi)$ and $(X,\pi')$ are minimal. Then $(X(\r),\pi)$ and $(X'(\r),\pi')$ are isomorphic if and only if there exists an automorphism of $\p^1_{\r}$ that sends $I(X,\pi)$ on $I(X',\pi')$. \qed
\end{cor}
\begin{cor}\label{Cor:Thm5Conic}
Let $(X,\pi_X)$ and $(Y,\pi_{Y})$ be two minimal conic bundles. Then, the following are equivalent:

\begin{enumerate}
\item \label{Cor:Thm5Conic1}$(X(\r),\pi_X)$ and $(Y(\r),\pi_Y)$ are isomorphic;
\item \label{Cor:Thm5Conic2}$(X,\pi_X)$ is birational to $(Y,\pi_Y)$ and $X(\r)$ is isomorphic to $Y(\r)$.\end{enumerate}
\end{cor}
\begin{proof}
The implication $\ref{Cor:Thm5Conic1})\Rightarrow \ref{Cor:Thm5Conic2})$ is evident. Let us prove the converse. 

Since $(X,\pi_{X})$ is birational to $(Y,\pi_{Y})$ and both of them are minimal, the number of singular fibres of $\pi_{X}$ and $\pi_{Y}$ is the same, equal to $2r$ for some non-negative integer~$r$.

Assume that $r=0$, which means that $X$ is an Hirzebruch surfaces $\mathbb{F}_{m}$ for some $m$ and that $Y=\mathbb{F}_{n}$ for some $n$. Since $X(\r)$ is isomorphic to $Y(\r)$, we have $m\equiv n \mod 2$. It is easy to prove that $(X(\r),\pi)$ and $(Y(\r),\pi)$ are isomorphic, by taking elementary links at two imaginary distinct fibres (see for example \cite[Proof of Theorem~6.1]{Ma06}). 

When $r>0$, already the fact that $(X,\pi_X)$ is birational to $(Y,\pi_Y)$ implies that $(X(\r),\pi_X)$ is isomorphic to $(Y(\r),\pi_Y)$ (Proposition~\ref{Prp:Equivalences}).
\end{proof}

\section{Conic bundles on del Pezzo surfaces}\label{Sec:ConicBunDPS}


In this section, we focus on surfaces admitting distinct minimal conic bundles. We will see that these surfaces are necessarily del Pezzo surfaces (Lemma~\ref{Lem:KX4}). We begin by the description of all possible minimal real conic bundles occurring on del Pezzo surfaces.

\begin{lem}\label{lem:minimaldp}
Let $V$ be is a subset of $\p^1(\r)$, then the following are equivalent: 
\begin{enumerate}
\item\label{lem:minimaldp1}
there exists a  minimal real conic bundle $(X,\pi)$ with $I(X,\pi)=V$ such that $X$ is a del Pezzo surface;
\item\label{lem:minimaldp2}
the set $V$ is a union of closed intervals, and $\#V\leq 3$.
\end{enumerate}
\end{lem}

\begin{proof}
The part $\ref{lem:minimaldp1})\Rightarrow \ref{lem:minimaldp2})$ is easy. Indeed, if $(X,\pi)$ is minimal, it is well-know that the number of singular fibres of $\pi$ is even, denoted $2r$, and that $2r = 8-(K_X)^2$. Since $-K_X$ is ample, $K_X^2\geq 1$, thus $r\leq 3$. The conclusion follows as $I(X,\pi)$ is the union of $r$ closed intervals.

Let us prove the converse. If  $V=\p^1(\r)$ or $V = \emptyset$, we take $(X,\pi)$ to be $(\p^1_\c\times\p^1_\c,\pr_1)$, where $\pr_1$ is the projection on the first factor, endowed with the anti-holomorphic map  that sends $\bigl((x_1:x_2),(y_1:y_2)\bigr)$ onto $\bigl((\overline{x_1}:\overline{x_2}),(\pm \overline{y_2}:\overline{y_1})\bigr)$.

Now we can assume that $V$ consists of $k$ closed intervals $I_1,\dots,I_k$, with $1\le k\le 3$. For $j=1,\dots, 3$, we denote by $m_j$ an homogenous form of degree $2$. If $j\le k$, we choose  that $m_j$ vanishes at the boundary of the interval $I_j$, and is non-negative on $I_j$. If $j>k$, we choose $m_j$ such that $m_j$ is positive on $\p^1(\r)$. In any case, we choose that $m_1\cdot m_2\cdot m_3$ has $6$ distinct roots. We consider the real surface given by
$$
X := \bigl\{\left((x:y:z),(a:b)\right)\in\mathbb{P}^2_\r\times\p^1_\r\ |\ x^2m_1(a,b)+y^2m_2(a,b)+z^2m_3(a,b)=0\bigr\}\;.
$$
The projection on $\p^2_\r$ is a double covering. A straightforward calculation shows that this covering is ramified over a smooth quartic. In consequence, $X$ is a smooth surface, and precisely a del Pezzo surface of degree $2$.
Taking $\pi\colon X\to \p^1_\r$  as the second projection, we obtain a conic bundle $(X,\pi)$ on the del Pezzo surface $X$ such that $I(X,\pi)=V$. If $k=3$, the conic bundle is minimal. Otherwise, we contract components in the imaginary singular fibres (corresponding to the roots of $m_j$ for $j>k$) to obtain the result.
\end{proof}
  
 Recall the following classical result, that will be useful throughout what follows.
  
  \begin{lem}\label{Lem:DPLinesP2}
  Let $\pi\colon S\to \p^1_{\C}$ be a complex conic bundle, and assume that $S$ is a del Pezzo surface, with $(K_S)^2=9-m\leq 7$. Then, there exists a birational morphism $\eta\colon S\to \p^2_\c$ which is a blow-up of $m$ points $p_1,\dots,p_m$ and which sends the fibres of $\pi$ onto the lines passing through $p_1$. The curves of self-intersection $-1$ of $S$ are
\begin{itemize}
\item  the exceptional curves $\eta^{-1}(p_1),\dots,\eta^{-1}(p_m)$; 
\item the strict transforms of the lines passing through $2$ of the $p_i$;
\item the conics passing through $5$ of the $p_i$;
\item the cubics passing through $7$ of the $p_i$ and being singular at one of these.

\end{itemize}
  \end{lem}
  
  \begin{proof}
  Denote by $\epsilon$ the contraction of one component in each singular fibre of $\pi$. Then, $\epsilon$ is a birational morphism of conic bundles -- not defined over $\r$ -- from $S$ to a del Pezzo surface which is also an Hirzebruch surface. Changing the contracted components, we may assume that $\epsilon$ is a map $S\rightarrow \mathbb{F}_1$. Contracting the exceptional section onto a point $p_1\in \p^2_\c$, we get a birational map $\eta\colon S\rightarrow \p^2_\c$ which is the blow-up of $m$ points $p_1,\dots,p_m$ of $\p^2_\c$, and which sends the fibres of $\pi_1$ onto the lines passing through $p_1$. The description of the $(-1)$-curves is well-known and may be found for example in \cite{bib:DemDPezzo}.
  \end{proof}
  
  \begin{lem}\label{Lem:KX4}
  Let $\pi_1\colon X\to \p^1_\r$ be a minimal real conic bundle. Then, the following conditions are equivalent:
  \begin{enumerate}
  \item\label{Lem:KX4-1}
 There exist a real conic bundle $\pi_2\colon X\rightarrow \p^1_{\r}$, such that $\pi_1$ and $\pi_2$ induce distinct foliations on $X(\c)$.
 \item\label{Lem:KX4-2}
   Either  $X$ is isomorphic to $\p^1_{\r}\times\p^1_{\r}$, or $X$ is a del Pezzo surface of degree $2$ or~$4$. 
     \end{enumerate}
     Moreover, if the conditions are satisfied, then the following occur:
    \begin{enumerate}
     \item[$a)$]
     The map $\pi_2$ is unique, up to an automorphism of $\p^1_\r$.
     \item[$b)$]
     There exist $\alpha \in \aut(X)$ and $\beta\in \aut(\p^1_\r)$ such that $\pi_1\alpha=\beta\pi_2$. Moreover, if $X$ is a del Pezzo surface of degree~$2$, $\alpha$ may be chosen to be the Geiser involution.
     \item[$c)$]
     Denoting by $f_1,f_2\subset \Pic(X)$ the divisors of the general fibre of respectively $\pi_1$ and $\pi_2$, we have $f_1+f_2= -cK_X$ where $c=4/(K_X)^2\in \n\cdot \frac{1}{2}$.
     \end{enumerate}
  \end{lem}
  \begin{proof}
  
We now prove that $\ref{Lem:KX4-1})$ implies $\ref{Lem:KX4-2})$, $a)$, and $c)$. 
 Assuming the existence of $\pi_2$, we denote by $f_i$ the divisor of the fibre of $\pi_i$ for $i=1,2$.
  We have $(f_1)^2=(f_2)^2=0$ and by adjunction formula $f_1\cdot K_X=f_2\cdot K_X=-2$, where $K_X$ is the canonical divisor. Let us write $d=(K_X)^2$.

Since $(X,\pi_1)$ is minimal, $\Pic(X)$ has rank $2$, hence $f_1=aK_X+bf_2$, for some $a,b\in \mathbb{Q}$. Computing $(f_1)^2$ and $f_1\cdot K_X$ we find respectively $0=a^2d-4ab=a(ad-4b)$ and $-2=ad-2b$. If $a=0$, we find $f_1=f_2$, a contradiction. Thus, $4b=ad$ and $2b=ad+2$, which yields $b=-1$ and $ad=-4$, so $f_1+f_2=-4/d\cdot K_X$. This shows that $f_2$ is uniquely determined by $f_1$, which is the assertion $a)$.

Denote as usual by $S$ the complex surface associated to $X$.
Let $C\in \Pic(S)$ be an effective divisor, with reduced support, and let us prove that $C\cdot (f_1+f_2)>0$. 
Since $C$ is effective, $C\cdot f_1\geq 0$ and $C\cdot f_2\geq 0$. 
If $C\cdot f_1=0$, then the support of $C$ is contained in one fibre of $\pi_1$. 
If $C$ is a multiple of $f_1$, then $C\cdot f_2>0$; otherwise, $C$ is a multiple of a $(-1)$-curve contained in a singular fibre of $f_1$, and the orbit of $C$ by the anti-holomorphic involution is equal to a multiple of $f_1$, whence $C\cdot f_2>0$. 

Since $f_1+f_2$ is ample, and $f_1+f_2=-4/d\cdot K_X$ either $K_X$ or $-K_X$ is ample. 
The surface $X$ being geometrically rational, the former cannot occur, whence $d>0$. 

If $S$ is isomorphic to $\p^1_{\c}\times\p^1_{\c}$,  the existence of $\pi_1,\pi_2$ shows that $X$ is isomorphic to $\p^1_{\r}\times\p^1_{\r}$. 
Otherwise, $K_X$ is not a multiple in $\Pic(X_{\c})$ and thus $d$ is equal to $1$, $2$ or $4$. 
The number of singular fibres being even and equal to $8-(K_X)^2$, the only possibilities are then $2$ and $4$.

We have proved that $\ref{Lem:KX4-1})$ implies $\ref{Lem:KX4-2})$, $a)$, and $c)$. 

Assume now that $X=(S,\sigma)$ is $\p^1_\r\times\p^1_\r$ or a del Pezzo surface of degree $2$ or~$4$. We construct an automorphism $\alpha$ of $X$ which does not belong to $\aut(X,\pi)$. Then, by taking $\pi_2=\pi_1\alpha$ we get assertion $\ref{Lem:KX4-1})$. Taking into account the unicity of $\pi_2$, we get $b)$.

If $X$ is $\p^1_{\r}\times\p^1_{\r}$, the two conic bundles are given by the projections on each factor, and we can get  for $\alpha$ the swap of the factors.

If $X$ is a del Pezzo surface of degree $2$, the anti-canonical map $\doublecov\colon X\rightarrow \p^2$ is a double covering ramified along a smooth quartic, cf. e.g. \cite{bib:DemDPezzo}. Let $\alpha$ be the involution associated to the double covering -- $\alpha$ is classically called the \emph{Geiser involution}. It fixes a smooth quartic, hence cannot preserve any conic bundle.

The remaining case is when $X$ is a del Pezzo surface of degree $4$. By Lemma~\ref{Lem:DPLinesP2}, there is a birational map $\eta\colon S\rightarrow \p^2_\c$ which is the blow-up of five points $p_1,\dots,p_5$ of $\p^2_\c$, no three being collinear and which sends the fibres of $\pi_1$ on the lines passing through $p_1$ . 
There are $16$ exceptional curves (curves isomorphic to $\p^1_\c$ of self-intersection $(-1)$) on $S$:
\begin{itemize}
\item  $E_1=\eta^{-1}(p_1),...,E_5=\eta^{-1}(p_5)$
  ($5$ curves);
  \item the strict transforms of the lines passing through $p_i$ and $p_j$, denoted by $L_{ij}$  ($10$ curves);
  \item the strict transform of the conic passing through the five points, denoted by $\Gamma$.
  \end{itemize}
  
  Note that the four singular fibres of $\pi_1$ are $E_i\cup L_{ij}$, $i=2,\dots,5$, and that $\sigma$ exchanges thus $E_i$ and $L_{ij}$ for $i=1,\dots,5$. The intersection form being preserved, this implies that $\sigma$ acts on the $16$ exceptional curves as
\[(E_2\ L_{12})(E_3\  L_{13})(E_4\  L_{14})(E_5\  L_{15})(E_1\  \Gamma)(L_{23}\  L_{45})(L_{24}\  L_{35})(L_{25}\  L_{34}).\]

After a linear change of coordinates, we may assume that $p_1=(1:1:1)$, $p_2=(1:0:0)$, $p_3=(0:1:0)$, $p_4=(0:0:1)$ and $p_5=(a:b:c)$ for some $a,b,c\in \c^{*}$. 
Denote by $\phi$ the birational involution $(x:y:z)\dasharrow (ayz:bxz:cxy)$ of $\p^2_\c$. 
Since the base-points of $\phi$ are $p_2,p_3,p_4$ and since $\phi$ exchanges $p_1$ and $p_5$, the map $\alpha=\eta^{-1}\phi\eta$ is an automorphism of $S$. 
Its action on the $16$ exceptional curves is given by the permutation
\[(L_{23}\ E_4)(L_{24}\  E_3)(L_{34}\  E_2)(L_{12}\  L_{25})(L_{13}\  L_{35})(L_{14}\  L_{45})(\Gamma\  L_{15})(E_1\  E_5).\]
Observe that the actions of $\alpha$ and $\sigma$ on the set of $16$ exceptional curves commute. This means that $\alpha \sigma \alpha^{-1}\sigma^{-1}$ is an holomorphic automorphism of $S$ which preserves any of the $16$ curves. 
It is the lift of an automorphism of $\p^2_\c$ that fixes the $5$ points $p_1,\dots,p_5$ and hence is the identity. 
Consequently, $\alpha$ and $\sigma$ commute, so $\alpha\in \aut(X)$. 
Since $\phi$ sends a general line passing though $p_1$ onto a conic passing through $p_2,\dots,p_5$, $\alpha$ belongs to $\aut(X)\backslash \aut(X,\pi)$.
  \end{proof}
  \begin{cor}\label{Cor:23comp2conicbundles}
Let $X$ be a  minimal geometrically rational real surface, which is not rational. Then, the following are equivalent:
\begin{enumerate}
\item\label{Cor:23comp2conicbundles1}
$\#X(\r)=2$ or $\#X(\r)=3$;
\item\label{Cor:23comp2conicbundles2}
There exists a geometrically rational real surface $Y(\r)$ isomorphic to $X(\r)$, and such that $Y$ admits two minimal conic bundles $\pi_1\colon Y\to \p^1_\r$ and $\pi_2\colon Y\to \p^1_\r$ inducing distinct foliations on $Y(\c)$. 
\end{enumerate}
  \end{cor}
  \begin{proof}
$[\ref{Cor:23comp2conicbundles2})\Rightarrow \ref{Cor:23comp2conicbundles1})]$ By Lemma~\ref{Lem:KX4}, $Y$ is then a del Pezzo surface, which has degree $2$ or $4$ since $Y$ is not rational. This implies that $\#Y(\r)=2$ or $\#Y(\r)=3$ by Proposition~\ref{Prp:TopMinimal}.

$[\ref{Cor:23comp2conicbundles1})\Rightarrow \ref{Cor:23comp2conicbundles2})]$.  According to Theorem~\ref{Thm:ClassicMinimal} and  Proposition~\ref{Prp:TopMinimal}, $(1)$ implies the existence of a minimal real conic bundle structure $\pi_{X}\colon X\to \p^1_{\r}$ with $4$ or $6$  singular fibres. This condition is equivalent to the fact that $I(X,\pi_{X})$ is the union of $2$ or $3$ intervals. According to Lemma~\ref{lem:minimaldp}, 
there exists a  minimal real conic bundle $(Y,\pi_{1})$ such that $Y$ is a del Pezzo surface and $I(Y,\pi_{1})=I(X,\pi_{X})$. Corollary~\ref{Cor:BirAutI} shows that $(X(\r),\pi_{X})$ and $(Y,\pi_{1})$ are isomorphic. Moreover
Lemma~\ref{Lem:KX4} yields the existence of $\pi_{2}$. 
  \end{proof}
  
\section{Equivalence of surfaces versus equivalence of conic bundles}\label{Sec:Equiv}
 
This section is devoted to the proof of Theorem~\ref{Thm:Diff+Bir=DiffBir}. From Theorem~\ref{Thm:ClassicMinimal} and Proposition~\ref{Prop:DP1DP2rho1}, it remains to solve the conic bundle case, which is done in Theorem~\ref{thm:isom}.
 First of all, we correct an existing inaccuracy in the literature; in \cite[Exercice 5.8]{Kol} or \cite[VI.3.5]{Sil}, it is asserted that all minimal real conic bundles with four singular fibres belong to a unique birational equivalence class.
To the contrary, the following general result, which includes the case with four singular fibres, occurs:

\begin{thm}\label{Thm:Birapport}
Let  $\pi_X\colon X\rightarrow \p^1_{\r}$ and $\pi_Y\colon Y\rightarrow \p^1_{\r}$ be two real conic bundles, and suppose that either $X$ or $Y$ is non-rational. Then, the following are equivalent:
\begin{enumerate}
\item\label{Thm:Birapport1}
 The two real surfaces $X$ and $Y$ are birational.
 \item\label{Thm:Birapport2}
 The two real conic bundles $(X,\pi_X)$ and $(Y,\pi_Y)$  are birational.
 \item\label{Thm:Birapport3}
There exists an automorphism of $\p^1$ which sends $I(X,\pi_X)$ onto $I(Y,\pi_Y)$.
  \end{enumerate}
  Moreover, if the number of singular fibres of $\pi_{X}$ is at least $8$, then $\bir(X)=\bir(X,\pi_{X})$.
  \end{thm}
  \begin{rem}
  It is well-known that this result is false when $X$ and $Y$ are rational. Indeed, consider $(X,\pi_X)=(\p^1_\r\times\p^1_\r,\pr_1)$ and $(Y,\pi_Y)$ be a real conic bundle with two singular fibres. The surfaces $X$ and $Y$ are birational, but the conic bundles $(X,\pi_X)$ and $(Y,\pi_Y)$ are not. 
  \end{rem}
  \begin{proof}
The equivalence $(\ref{Thm:Birapport3})\Leftrightarrow (\ref{Thm:Birapport2})$ was proved in Corollary~\ref{Cor:BirAutI} and  $(\ref{Thm:Birapport2})\Rightarrow (\ref{Thm:Birapport1})$ is evident.

Let us prove now $(\ref{Thm:Birapport1})\Leftrightarrow (\ref{Thm:Birapport2})$. We may assume that $(X,\pi_X)$ and $(Y,\pi_Y)$ are minimal and that $X$ is not rational, hence $\pi_X$ has at least $4$ singular fibres. Let  $\psi\colon X\dasharrow Y$ a birational map, and decompose $\psi$ into elementary links: $\psi=\psi_n\circ\dots\circ\psi_1$ (see \cite[Theorem 2.5]{IskFact}). Consider $\psi_1 \colon X\dasharrow X_1$ the first link, which may be of type $II$ or $IV$ only by \cite[Theorem 2.6]{IskFact}. If $\psi_1$ is of type $II$, then $\psi_1$ is a birational map of conic bundles  $(X,\pi_X) \dasharrow (X_1,\pi_1)$ for some conic bundle structure $\pi_1\colon X_1\rightarrow \p^1$. If $\psi_1$ is of type $IV$, then $\psi_1$ is an isomorphism $X \to X_1$ and the link is precisely a change of conic bundle structure from $\pi_{X}$ to $\pi_1\colon X_1\rightarrow\p^1$, which induce distinct foliations on $X(\r)$. 
Applying Lemma~\ref{Lem:KX4}, $X$ is a del Pezzo surfaces of degree $2$ or $4$, and there exist automorphisms $\alpha\in \aut(X)$ and $\beta\in \aut(\p^1_\r)$ such that $\pi_1\psi_1\alpha=\beta\pi_2$, whence $(X,\pi)$ is isomorphic to $(X_1,\pi_1)$. We  proceed by induction on the number of elementary links to conclude that $(X,\pi_X)$ is birational to $(Y,\pi_Y)$. Moreover, if $\pi_{X}$ has at least $8$ singular fibres, then no link of type $IV$ may occur, so $\psi$ is a birational map of conic bundles $(X,\pi_{X})\dasharrow (Y,\pi_{Y})$. 
  \end{proof}
 When the conic bundles are minimal, we can strengthen Theorem~\ref{Thm:Birapport} to get an isomorphism between the real parts.
  
  \begin{thm}\label{thm:isom}
    Let  $\pi_X\colon X\rightarrow \p^1_{\r}$ and $\pi_Y\colon Y\rightarrow \p^1_{\r}$ be two \emph{minimal} real conic bundles, and suppose that either $X$ or $Y$ is non-rational. Then, the following are equivalent:
\begin{enumerate}
\item\label{enum.thm.isom.1}
$X$ and $Y$ are birational.
\item\label{enum.thm.isom.2}
 $X(\r)$ and $Y(\r)$  are isomorphic.
\item\label{enum.thm.isom.3}
 $(X(\r),\pi_X)$ and $(Y(\r),\pi_Y)$  are isomorphic.
  \end{enumerate}
  \end{thm}
  \begin{proof}
The implications \ref{enum.thm.isom.3}) $\Rightarrow$ \ref{enum.thm.isom.2}) $\Rightarrow$ \ref{enum.thm.isom.1}) being evident, it suffices to prove  \ref{enum.thm.isom.1}) $\Rightarrow$  \ref{enum.thm.isom.3}).
Since $X$ and $Y$ are not rational, both $\pi_{X}$ and $\pi_{Y}$ have at least one singular fibre. Applying Lemma~\ref{Lem:IsoEx}, we may assume that both $(X,\pi_X)$ and $(Y,\pi_Y)$ are exceptional real conic bundles. Then, since $(X,\pi_X)$ and $(Y,\pi_Y)$ are birational (Theorem~\ref{Thm:Birapport}), we may assume that $I(X,\pi_X)=I(Y,\pi_Y)$, up to an automorphism of $\p^1_\r$. Then Corollary~\ref{Cor:BirAutI} shows that $(X,\pi_X)$ is isomorphic to $(Y,\pi_Y)$.
  \end{proof}

We are now able to prove Theorem~\ref{Thm:Diff+Bir=DiffBir} concerning minimal surfaces. 
\begin{proof}[Proof of Theorem~$\ref{Thm:Diff+Bir=DiffBir}$]
Let  $X$ and $Y$ be two minimal geometrically rational real surfaces, and assume that either $X$ or $Y$ is non-rational. 

If $X(\r)$ and $Y(\r)$  are isomorphic, it is clear that $X$ and $Y$ are birational. Let us prove the converse.

 Theorem~\ref{Thm:ClassicMinimal} lists all the possibilities for $X$. If $\rho(X)=1$ or $\rho(Y)=1$, Proposition~\ref{Prop:DP1DP2rho1} shows that $X$ is isomorphic to $Y$. Otherwise, since neither $X$ nor $Y$ is rational, there exist minimal conic bundle structures on $X$ and on $Y$. From Theorem~\ref{thm:isom}, we conclude that $X(\r)$ is isomorphic to $Y(\r)$.
\end{proof}
To go further with non-minimal surfaces, we need to know when the group  $\aut\bigl(X(\r)\bigr)$ is very transitive for $X$ minimal.  This is done in the next sections.

\section{Very transitive actions}\label{Sec:VeryTransitive}

Thanks to the work done in Section~\ref{Sec:conicbundle}, it is easy to apply the techniques of \cite{hm3} to prove that 
$\aut\bigl(X(\r)\bigr)$ is fiberwise very transitive on a real conic bundle. After describing the transitivity of $\aut\bigl(X(\r)\bigr)$ on the tangent space of a general point, we set the main result of that section: $\aut\bigl(X(\r)\bigr)$ is very transitive on each connected component when  $X$ is minimal and admits two conic bundle structures (Proposition~\ref{Prp:TransivitySurface}). We end the section by giving a characterisation of surfaces $X$ for which $\aut\bigl(X(\r)\bigr)$ is able to mix the connected components of $X(\r)$.

\begin{lem}\label{Lem:TransivityConicBundle}
Let $(X,\pi)$ be a minimal real conic bundle over $\p^1_\r$ with at least one singular fibre. 
Let $(p_1,\dots,p_n)$ and $(q_1,\dots,q_n)$ be two $n$-tuples of distinct points of $X(\r)$, and let $(b_1,\dots,b_m)$ be $m$ points of $I(X,\pi)$. Assume that $\pi(p_i)=\pi(q_i)$ for each $i$, that $\pi(p_i)\ne\pi(p_j)$ for $i\ne j$ and that $\pi(p_i)\not=b_j$ for any $i$ and any $j$. 

Then, there exists $\alpha \in \aut\bigl(X(\r)\bigr)$ such that $\alpha(p_i)=q_i$ for every $i$, $\pi\alpha=\pi$ and $\alpha|_{\pi^{-1}(b_i)}$ is the identity for every $i$.
\end{lem}
\begin{rem}
The same result holds for minimal real conic bundles with no singular fibre, see \cite[5.4]{bh}. The following proof uses \emph{twisting maps}, see below, which were introduced in \cite{hm3} to prove that the action of the group of automorphisms $\aut(S^2)$ on the quadric sphere $S^2:=\{(x:y:z)\in\r^3\ |\ x^2+y^2+z^2=1\}$ is  very transitive.
\end{rem}
\begin{proof}
By Lemma~\ref{Lem:IsoEx}, we may assume that $(X,\pi)$ is exceptional. 
Moreover, Lemma~\ref{Lem:AffEx} yields the existence of an affine real surface $A\subset X$ isomorphic to the hypersurface of $\r^3$ given by \[y^2+z^2=-\prod_{i=1}^{2r}(x-a_i),\]
for some $a_1,\dots,a_{2r}\in \r$ with $a_1<a_2<\dots<a_{2r}$, where $\pi|_A$ corresponds to the projection $(x,y,z)\mapsto x$, and where the inclusion $A\subset X$ induces an isomorphism $A(\r)\rightarrow X(\r)$.

 For $i=1,\dots,n$, let us denote by $(x_i,y_i,z_i)$ the coordinates of $p_i$ in $A\subset \r^3$ and by $(u_i,v_i,w_i)$ the ones of $q_i$. 
From hypothesis, we have  $x_i = u_i$ for all $i$, thus we get $y_i^2 + z_i^2= v_i^2 + w_i^2$ for all $i$. Let $\Phi_i\in \SO_2(\r)$ be the rotation sending $(x_i,y_i)$ to $(u_i,v_i)$. Then by \cite[Lemma~2.2]{hm3}, there exists an algebraic map $\Phi
\colon [a_1,a_{2r}] \to \SO_2(\r)$ such that $\Phi(x_i) =
\Phi_i$ for $i=1,\dots,n$ and $\Phi(b_i)$ is the identity for $i=1,\dots,m$. Let us recall the proof; since $ \SO_2(\r)$ is isomorphic to the unit circle
$S^1:=\{(x:y:z)\in\p^2(\r)\ |\ x^2+y^2=z^2\}$, it suffices to prove the statement for $S^1$ instead of
$ \SO_2(\r)$. Let~$\Phi_0$ be a point of~$S^1$ distinct
from~$\Phi_1,\ldots,\Phi_n$ and from the identity.  Since $S^1\setminus \{\Phi_0\}$ is
isomorphic to $\r$, it suffices, finally, to prove
the statement for $\r$ instead of $ \SO_2(\r)$. The
latter statement is an easy consequence of Lagrange polynomial
interpolation.

Then the map defined by $\alpha\colon (x,y,z)\mapsto \bigl(x,(y,z)\cdot \Phi(x)\bigr)$ induces an automorphism $A(\r)\rightarrow A(\r)$ called the \emph{twisting map} of $\pi$ associated to $\Phi$. Moreover, $\alpha(p_i)=q_i$, for all $i$, $\pi\alpha=\pi$, $\alpha|_{\pi^{-1}(b_i)}$ is the identity for every $i$, and $\pi$ induces an automorphism $X(\r)\to X(\r)$.
\end{proof}

\begin{lem}\label{Lem:TransivityTangentConicBundle}
Let $(X,\pi)$ be a minimal real conic bundle over $\p^1_\r$ with at least one singular fibre. Let $p\in X$ be a real point in a nonsingular fibre of $\pi$, and let $\Sigma\subset I(X,\pi)$ be a finite subset, with $\pi(p)\in \Sigma$.
 Denote by $\eta\colon Y\to X$ the blow-up of $p$, and by $E\subset Y$ the exceptional curve. Let $q\in E$ the point corresponding to the direction of the fibre of $\pi$ passing through $p$.

Then, the lift of the group \[G=\Big\{\alpha \in \aut\bigl(X(\r)\bigr), \pi\alpha=\pi\ \Big|\ \alpha |_{\pi^{-1}(\Sigma)} \mbox{ is the identity}\Big\}\]by $\eta$ is a subgroup $\eta^{-1}G\eta\subset \aut\bigl(Y(\r)\bigr)$ which fixes the point $q$, and acts transitively on $E\backslash q\cong \mathbb{A}^1_\r$.
\end{lem}
\begin{proof}Since $G$ acts identically on $\pi^{-1}(\Sigma)$, it fixes $p$, and therefore lifts to $H=\eta^{-1} G\eta \subset \aut(Y(\r),\pi\eta)$, which preserves $E$. Moreover, $G$ preserves the fibre of $\pi$ passing through $p$, so $H$ preserves its strict transform, which intersects transversally $E$ at $q$, so $q$ is fixed.

Let us prove now that the action of $\eta^{-1}G\eta$ on $E\backslash q$ is transitive. 
By Lemma~\ref{Lem:IsoEx}, we may assume that $(X,\pi)$ is exceptional. 
Then, we take an affine surface $A\subset X$, isomorphic to the hypersurface $y^2+z^2=P(x)$ of $\r^{3}$ for some polynomial $P$, such that $A|_{\pi}$ is the projection $\pr_{x}\colon(x,y,z)\mapsto x$ and the inclusion $A\subset X$ gives an isomorphism $A(\r)\to X(\r)$ (Lemma~\ref{Lem:AffEx}). 
Let us write $(x_{0},y_{0},z_{0})\in \r^{3}$ the coordinates of $p$. Since $x$ is on a nonsingular fibre of $\pi$, then $P(x_{0})>0$. Up to an affine automorphism of $\r^{3}$, and up to multiplication of $P$ by some constant, we may assume that $x_{0}=0$, $P(0)=1$, $y_{0}=0$, and $z_{0}=0$.

To any real polynomial $\lambda\in \r[X]$, we associate the matrix 
$$
\left(\begin{array}{cc}\alpha(X)& \beta(X)\\ -\beta(X) & \alpha(X)\end{array}\right)\in \SO_{2}(\r(X))\;,
$$
where $\alpha=\frac{1-\lambda^{2}}{1+\lambda^{2}}\in\r(X)$ and $\beta=\frac{2\lambda}{1+\lambda^2}\in\r(X)$. And corresponding to this matrix, we associate the map 
\[
\psi_{\lambda}\colon (x,y,z)\mapsto (x,\alpha(x)\cdot y-\beta(x)\cdot z,\beta(x)\cdot y+\alpha(x)\cdot z),
\]
which belongs to $\aut(A(\r),\pr_{x})$. 
To impose that $\psi_{\lambda}$ is the identity on $(\pr_{x})^{-1}(\Sigma)$ is the same to ask that $\lambda(x)=0$ for each $(x:1)\in \Sigma\subset \p^1(\r)$, and in particular for $x=0$. 

Denote by $\mathcal{O}=\r[x,y,z]/(y^2+z^2-P(x))$ the ring of functions of $A$, by $\mathfrak{p}\subset \mathcal{O}$ the ideal of functions vanishing at $p$, by $\mathcal{O}_{\mathfrak{p}}$ the localisation, and by $\m \subset \mathcal{O}_{\mathfrak{p}}$ the maximal ideal of $\mathcal{O}_{\mathfrak{p}}$. 
Then,  the cotangent ring $T_{p,A}^{*}$ of $p$ in $A$ is equal to $\m/\m^2$, and is generated by the images $[x]$, $[y]$, $[z-1]$ of $x,y,z-1\in \r[x,y,z]$. Since $P(0)=1$, we may write $P(x)=1+xQ(x)$, for some real polynomial $Q$. 
We compute 
\[
[0]=[y^2+z^2-P(x)]=[y^2+(z-1)^2+2(z-1)-xQ(x)]=[2(z-1)-xQ(0)]\in\m/\m^2\;.
\]  

We see that $[z-1]=[xQ(0)/2]$, thus $\m/\m^2$ is generated by $[x]$ and $[y]$ as a $\r$-module. Since $\lambda(0)=0$, we can write $\lambda(x)=x\mu(x)$, for some real polynomial $\mu$. The linear action of $\psi_{\lambda}$ on the cotangent space $T_{p,A}^{*}$ fixes $[x]$  and sends $[y]$ onto 
\[\begin{array}{rcl}
\left[\alpha(x)\cdot y-\beta(x)\cdot z\right]&=&\left[\frac{(1-\lambda(x)^2)y-2\lambda(x)z}{\lambda(x)^2+1}\right]=\left[y-2\lambda(x)(1+xQ(0)/2)\right]\\
&=&\left[y-2\mu(0)x\right]\;.\end{array}
\]
It suffices to change the derivative of $\lambda$ at $0$ (which is equal to $\mu(0)$), which may be any real number. Therefore, the action of $G$ on the projectivisation of $T_{p,A}^{*}$, fixes a point (corresponding to $[x]$) but acts transitively on the complement of this point. Since $E$ corresponds to the projectivisation of $T_{p,A}$,  $G$ acts transitively on $E\backslash q$.\end{proof}

\begin{lem}\label{no-two-singular}
Let $X$ be a real projective surface endowed with two minimal conic bundles  $\pi_1\colon X\to \p^1_\r$ and $\pi_2\colon X\to \p^1_\r$  inducing distinct foliations on $X(\c)$. 
There exists a real projective surface ${X'}$ such that ${X'}(\r)$ and $X(\r)$ are isomorphic, $X$ is endowed with two minimal conic bundles  $\pi_1'\colon X\to \p^1_\r$ and $\pi_2'\colon X\to \p^1_\r$  inducing distinct foliations on $X'(\c)$ 
 and the following condition holds:

$(\star)$ Let $F_j$ be a real fibre of $\pi_j'$, $j=1,2$. If $F_1(\r)\cap F_2(\r) \ne \emptyset$, then at most one of the curves $F_j$ can be singular. 
\end{lem}
\begin{rem}
It is possible that the condition $(\star)$ above does not hold for $(X, \pi_1,\pi_2)$, taking for example for ${X}$ the del Pezzo surface of degree $2$ given in the proof of Lemma~\ref{lem:minimaldp} for $k=3$:
$$
X := \bigl\{\left((x:y:z),(a:b)\right)\in\mathbb{P}^2_\r\times\p^1_\r\ |\ x^2m_1(a,b)+y^2m_2(a,b)+z^2m_3(a,b)=0\bigr\}\;.
$$
The map $\pi_1\colon {X}\to \p^1_\r$ is given by the second projection, and the $6$ singular points of its singular fibres correspond to only three points of $\p^2_\r$, namely $(1:0:0)$, $(0:1:0)$ and $(0:0:1)$. This shows that the Geiser involution preserves the set of the $6$ points, so each of these points is the singular point of a singular fibre of $\pi_2$.
\end{rem}
\begin{proof}Suppose that the condition $(\star)$ does not hold for $(X,\pi_1,\pi_2)$ (otherwise, the result is obvious). Then $F_i$ is the union of two $(-1)$-curves $E_{i,1}$ and $E_{i,2}$, intersecting transversally at some point $p_i$. Since $p_i$ is the only real point of $F_i$, we have $p_1=p_2$. Hence, $F_1\cdot E_{2,i}\geq 2$ for $i=1,2$, which implies that $F_1\cdot F_2\geq 4$. According to Lemma~\ref{Lem:KX4}, $X$ is a del Pezzo surface of degree $2$ or $4$, and we have $F_1+F_2=-cK_{X}$ with $c=4/(K_{X})^2$. Computing $16/(K_{X})^2=(F_1+F_2)^2=2F_1\cdot F_2\geq 8$, we see that $(K_{X})^2=2$.

Let $q\in X(\r)$ be a real point, let $\eta\colon Y\to X$ be the blow-up of $q$, and let $\epsilon\colon Y\to {X'}$ be the contraction of the strict transform of the fibre of $\pi_1$ passing through $q$, and write $\psi\colon X\dasharrow {X'}$ the composition $\psi=\epsilon\circ \eta^{-1}$. We prove now that if $q$ is general enough, then ${X'}$ is a del Pezzo surface of degree $2$, and  $\pi_1'=\pi_1\circ \psi$ and $\pi_2'=\sigma_{X'}\circ \pi_1'$ (where $\sigma_{X'}\in \aut({X'})$ is the Geiser involution of ${X'}$) satisfy the condition $(\star)$.

Firstly, it is well-known that blowing-up a general point of a del Pezzo surface of degree $2$ yields a del Pezzo surface of degree $1$ (it suffices that $q$ does not belong to any of the $(-1)$-curves of $X$ and to the ramification curve of the double covering $X\to \p^2$); then a contraction from a del Pezzo surface of degree $1$ yields a del Pezzo surface of degree $2$.

Secondly, we denote respectively by $S,S',T$ the complex surfaces  obtained by forgetting the real structures of $X,X',Y$ and study condition $(\star)$ by working now in the Picard groups of these surfaces, identifying a curve with its equivalence class. We choose a $(-1)$-curve (not defined over~$\r$) in any of the six singular fibres of $\pi_1'$, and denote these by $C_1,\dots,C_6$, and denote by $p_1,\dots,p_6$ the singular points of the six singular fibres, so that $p_i\in C_i$. Condition $(\star)$ amounts to prove that $D_i:=\psi^{-1}\sigma_{X'}\psi(C_i)\subset S$ does not pass through $p_j$ for any $i$ and any $j$. Fixing $i$ and $j$, we will see that this yields a curve of $X$ where $q$ should not lie.  Note that the action of the Geiser involution $\sigma_{X'}\in \aut({X'})\subset\aut(S')$ on $\Pic(S')$ is given $\sigma_{X'}(D)=(D\cdot K_{X'})K_{X'}-D$ (follows directly from the fact that the invariant part of $\Pic(S')$ has rank $1$). In consequence, the $(-1)$-curve $D_i':=\sigma_{X'}\psi(C_i)\subset S'$ is equal to 
 $-K_{X'}-\psi(C_i)$, and thus $\epsilon^{*}(D_i)=-\epsilon^{*}(K_{X'})-\eta^{*}(C_i)$. Writing $E_q$ the $(-1)$-curve contracted by $\eta$, and $f$ a general fibre of $\pi_1$, the $(-1)$-curve contracted by $\epsilon$ is equivalent to $\eta^{*}(f)-E_q$. We have $K_Y=\eta^{*}(K_{X})+E_q=\epsilon^{*}(K_{X'})+\eta^{*}(f)-E_q$ in $\Pic(Y)$. This implies that 
$$\eta^{*}(D_i)=\epsilon^{*}(D_i')=-\eta^{*}(K_{X})+\eta^{*}(f)-\eta^{*}(C_i)-2E_q\in\Pic(Y).$$

This means that $D_i$ is a curve with a double point at $q$, is equivalent to $-K_{X}+f-C_i\in\Pic(S)$ and has self-intersection $3$. Moreover, the linear system $\Lambda_i$ of curves in $\Pic(S)$ equivalent to $-K_{X}+f-C_i$  has dimension $3$. Note that $\Lambda_i$ does not depend on $q$, but only on $i$. Denote by $\Lambda_{i,j}\subset \Lambda_i$ the sublinear system of curves of $\Lambda_i$ passing through $p_j$. This system has dimension $2$; after blowing-up $p_j$, the system $\Lambda_{i,j}$  yields  a ramified double covering of $\p^2$. If $D_i$ passes through $p_j$, then $D_i$ corresponds to a member of $\Lambda_{i,j}$, singular at $q$ and this implies that $q$ belongs to the ramified locus of the double covering induced by $\Lambda_{i,j}$. It suffices to choose $q$ outside of all these locus to obtain condition $(\star)$. 
\end{proof}
We now use the above lemmas to show that the action of $\aut\bigl(X(\r)\bigr)$ is very transitive on each connected component when $X$ is a surface with two conic bundles.

\begin{prop}\label{Prp:TransivitySurface}
Let $X$ be a real projective surface, which admits two minimal conic bundles $\pi_1\colon X\to \p^1_\r$ and $\pi_2\colon X\to \p^1_\r$ inducing distinct foliations on $X(\c)$. 

Let $(p_1,\dots,p_n)$ and $(q_1,\dots,q_n)$ be two $n$-tuples of distinct points of $X(\r)$ such that $p_i$ and $q_i$ belong to the same connected component for each $i$. Then, there exists an element of $\aut\bigl(X(\r)\bigr)$ which sends $p_i$ on $q_i$ for each $i$, and which sends each connected component of $X(\r)$ on itself.
\end{prop}

\begin{proof}

When $X$ is rational, the result follows from \cite[Theorem~1.4]{hm3}. Thus we assume that $X$ is non-rational, and in particular that $X(\r)$ is non-connected.

From Lemma~\ref{no-two-singular}, we can assume that any real point which is critical for one fibration is not critical for the second fibration. Otherwise speaking (recall that the fibrations are minimal) a real intersection point of a fibre $F_1$ of $\pi_1$ with a fibre $F_2$ of $\pi_2$ cannot be a singular point of $F_1$ and of $F_2$ at the same time. By Lemma~\ref{Lem:TransivityConicBundle} applied to $(X,\pi_1)$, and to $(X,\pi_2)$, we may assume without loss of generality that all 
points $p_1,\dots,p_n,q_1,\dots,q_n$ belong to smooth fibres of $\pi_1$ and to smooth fibres of $\pi_2$. 
We now use Lemma~\ref{Lem:TransivityConicBundle} to obtain an automorphism $\alpha$ of $(X(\r),\pi_1)$ such that $\pi_2(\alpha(p_i))\ne \pi_2(\alpha(p_j))$ and $\pi_2(\alpha(q_i))\ne \pi_2(\alpha(q_j))$ for $i\ne j$.
Hence, we may suppose that $\pi_2(p_i)\ne \pi_2(p_j)$ and $\pi_2(q_i)\ne \pi_2(q_j)$ for $i\ne j$.

Likewise, using an automorphism of $(X(\r),\pi_2)$ we may suppose that $\pi_1(p_i)\ne \pi_1(p_j)$ and $\pi_1(q_i)\ne \pi_1(q_j)$ for $i\ne j$.

We now show that for $i=1,\dots,m$, there exists an element $\alpha_i\in \aut\bigl(X(\r)\bigr)$ that sends $p_i$ on $q_i$ and that restricts to the identity on the sets $\cup_{j\not= i}\{p_j\}$ and $\cup_{j\not= i}\{q_j\}$. 
Then, the composition of the $\alpha_i$ will achieve the proof. Observe that $\doublecov=\pi_1\times \pi_2$ gives a finite surjective morphism $X\rightarrow \p^1_\r\times\p^1_\r$ which is $2$-to-$1$ or $4$-to-$1$ depending of the degree of $X$ (follows from assertion $(c)$ of Lemma~\ref{Lem:KX4}). 
Denote by $W$ the image of $X(\r)$. The map $X(\r)\to W$ is a differential map, which has topological finite degree. 
Denote by $W_i$ the connected component of $W$ which contains both $\doublecov(p_i)$ and $\doublecov(q_i)$. 
Observe that $W_{i}$ is contained in the square $I(X,\pi_{1})\times I(X,\pi_{2})$, and that for each point $x\in W_{i}$, the intersection of the horizontal and vertical lines (fibres of the two projections of $\p^1_{\r}\times\p^1_{\r}$) passing through $x$ with $W_{i}$ is either only  $\{x\}$, when $x$ is on the boundary of $W_{i}$, or is a bounded interval. Moreover, $W_{i}$ is connected. 
Then, there exists a path from $\doublecov(p_i)$ to $\doublecov(q_i)$ which is a sequence of vertical or horizontal segments contained in $W_i$. We may furthermore assume that none of the segments is contained in $(\pr_1)^{-1} (\pi_1(a))$ or $(\pr_2)^{-1} (\pi_2(a))$ for any $a\in (\cup_{j\not= i}\{p_j\})\cup( \cup_{j\not= i}\{q_j\})$. 
Denote by $r_1,...,r_l$ the points of $U$ that are sent on the singular points or ending points of the path, and by $s_1,\dots,s_l$ some points of $X(\r)$ which are sent by $\doublecov$ on $r_1,\dots,r_l$ respectively. Up to renumbering, $s_1=p_i, s_l=q_i$ and two consecutive points $s_j$ and $s_{j+1}$ are such that $\pi_1(s_j)=\pi_1(s_{j+1})$ or $\pi_2(s_j)=\pi_2(s_{j+1})$. We construct then $\alpha_i$ as a composition of $l-1$ maps, each one belonging either to $\aut(X(\r),\pi_1)$ or $\aut(X(\r),\pi_2)$ and sending $s_j$ on $s_{j+1}$, and fixing the points $(\cup_{j\not= i}\{p_j\})\cup( \cup_{j\not= i}\{q_j\})$. 
\end{proof}

The following proposition describes the possible mixes of connected components.

\begin{prop}\label{Prp:ExchangeComp}
Let $(X,\pi)$ be a minimal real conic bundle. Denote by $I_{1},\dots,I_{r}$ the $r$ connected components of $I(X,\pi)$, and by $M_{1},\dots,M_{r}$ the $r$ connected components of $X(\r)$, where $I_{i}=\pi(M_{i})$, $M_{i}=\pi^{-1}(I_{i})\cap X(\r)$. If $\nu\in \Sym_{r}$ is a permutation of $\{1,\dots,r\}$, the following are equivalent:
\begin{enumerate}
\item\label{Prp:ExchangeComp1}
there exists $\alpha\in \aut(\p^1_{\r})$ such that $\alpha(I_{i})=I_{\nu(i)}$ for each $i$;
\item\label{Prp:ExchangeComp2}
there exists $\beta\in \aut(X(\r),\pi)$ such that $\beta(M_{i})=M_{\nu(i)}$ for each $i$;
\item\label{Prp:ExchangeComp3}
there exists $\beta\in \aut\bigl(X(\r)\bigr)$ such that $\beta(M_{i})=M_{\nu(i)}$ for each $i$;
\item\label{Prp:ExchangeComp4}
there exist two real Zariski open sets $V,W\subset X$, and $\beta\in \bir(X)$, inducing an isomorphism $V\to W$,  such that $\beta(V(\r)\cap M_{i})=W(\r)\cap M_{\nu(i)}$ for each $i$.
\end{enumerate}
Moreover, the conditions are always satisfied when $r\leq 2$, and are in general not satisfied when $r\geq 3$.
\end{prop}

\begin{proof}
The implications $(\ref{Prp:ExchangeComp2})\Rightarrow (\ref{Prp:ExchangeComp1})$ and $(\ref{Prp:ExchangeComp2})\Rightarrow (\ref{Prp:ExchangeComp3})\Rightarrow (\ref{Prp:ExchangeComp4})$ are obvious.

The implication $(\ref{Prp:ExchangeComp1})\Rightarrow (\ref{Prp:ExchangeComp2})$ is a direct consequence of Corollary~\ref{Cor:BirAutI}
). 

We prove now that if $r\leq 2$, Assertion $(\ref{Prp:ExchangeComp1})$ is always satisfied, hence all the conditions are equivalent (since all are true). When $r\leq 1$, take $\alpha$ to be the identity. When $r=2$, we make a linear change of coordinates to the effect  that  $I_{1}=\{(x:1)\ |\ 0\le x\le 1\}$ and $I_{2}$ is bounded by $(1:0)$ and $(\lambda:1)$, for some $\lambda \in \r$, $\lambda>1$ or $\lambda<0$. Then, $\alpha\colon (x_{1}:x_{2})\mapsto (\lambda x_{2}:x_{1})$ is an involution which exchanges $I_{1}$ and $I_{2}$.

It remains to prove the implication $(\ref{Prp:ExchangeComp4})\Rightarrow (\ref{Prp:ExchangeComp1})$ for $r\geq 3$. We decompose $\beta$ into elementary links \[X=X_{0}\stackrel{\beta_{1}}{\dasharrow}X_{1}\stackrel{\beta_{2}}{\dasharrow}\dots \stackrel{\beta_{n-1}}{\dasharrow}X_{n-1}\stackrel{\beta_{n}}{\dasharrow}X_{n}=X\]
as in \cite[Theorem 2.5]{IskFact}. It follows from the description of the links of \cite[Theorem 2.6]{IskFact} that each of the links is of type $II$ or $IV$, and that the links of type $II$ are birational maps of conic bundles and the links of type $IV$ occur on del Pezzo surfaces of degree $2$. 

In consequence, each of the $X_{i}$ admits a conic bundle structure given by $\pi_{i}\colon X_{i}\to \p^1_{\r}$, where $\pi_{0}=\pi_{n}=\pi$, and if $\beta_{i}$ has type $II$, it  is a birational map of conic bundles $(X_{i-1},\pi_{i-1})\dasharrow (X_{i},\pi_{i})$, and if it has type $IV$, it is an isomorphism $X_{i-1}\to X_{i}$ which does not send the general fibre of $\pi_{i-1}$ on those of $\pi_{i}$. In this latter case, since $\pi_{i}$ and $\pi_{i-1}\beta_{i}$ have distinct general fibres, $X_{i-1}$ and $X_{i}$ are del Pezzo surfaces of degree~$2$, and the Geiser involution $\iota_{i-1}\in \aut(X_{i-1})$ exchanges the two general fibres (follows from \cite[Theorem 2.6]{IskFact}, but also from Lemma~\ref{Lem:KX4}). 
This means that the map $\beta_{i}\circ\iota_{i-1}$, that we denote by $\gamma_{i}$, is an isomorphism of conic bundles $(X_{i-1},\pi_{i-1})\to (X_{i},\pi_{i})$.

 Now, we prove by induction on the number of links of type $IV$ that $\beta$ may be decomposed into compositions of elements of $\bir(X,\pi)$ and maps of the form $\psi \iota \psi^{-1}$ where $\psi$ is a birational map of conic bundles $(X,\pi)\dasharrow (X',\pi')$, $(X',\pi')$ is a del Pezzo surface of degree $2$ and $\iota\in \aut(X')$ is the Geiser involution. 
 If there is no link of type $IV$, $\beta$ preserves the conic bundle structure given by $\pi$. Otherwise, denote by $\beta_{i}$ the first link of type $IV$, which is an isomorphism $\beta_{i}\colon X_{i}\to X_{i+1}$, and write $\beta_{i}=\gamma_{i}\circ \iota_{i-1}$ as before. 
 We write $\psi=\beta_{i-1}\circ\dots\circ\beta_{1}$, which is a birational map of conic bundles $\psi\colon (X,\pi)\dasharrow (X_{i},\pi_{i})$. 
 Then, $\beta= (\beta_{n}\circ\dots\circ\beta_{i+1}\circ\gamma_{i}\circ\psi)(\psi^{-1}\iota_{i-1}\psi)$. Applying the induction hypothesis on the map $(\beta_{n}\circ\dots\circ\beta_{i+1}\circ\gamma_{i}\circ\psi)\in \bir(X)$, we are done.

Now, observe that when $(X',\pi')$ is a minimal real conic bundle and $X'$ is a del Pezzo surface of degree $2$, the map $\doublecov\colon X'\to\p^2_{\r}$ given by $|-K_{X'}|$ is a double covering, ramified over a smooth quartic curve $\Gamma\subset \p^2_{\r}$ (see e.g.\ \cite{bib:DemDPezzo}). 
Since $(X,\pi)$ is minimal, $(K_X)^2=8-2r$ thus $\pi$ has $r=6$ singular fibres , so $I(X,\pi)$ is the union of three intervals and $X(\r)$ is the union of $3$ connected components. 
This implies that $\Gamma(\r)$ is the union of three disjoint ovals. A connected component $M$ of $X(\r)$ is homeomorphic to a sphere, and  surjects by $\doublecov$ to the interior of one of the three ovals.
 The  Geiser involution (induced by the double covering) induces an involution on $M$, which fixes the preimage of the oval. 
 This means that the Geiser involution sends any connected component of $X(\r)$ on itself. 
 Thus, in the decomposition of $\beta$ into elements of $\bir(X,\pi)$ and conjugate elements of Geiser involutions, the only relevant elements are those of $\bir(X,\pi)$. There exists thus $\beta'\in \bir(X,\pi)$ which acts on the connected components of $X(\r)$ in the same way as $\beta$. This shows that $(\ref{Prp:ExchangeComp4})$ implies $(\ref{Prp:ExchangeComp1})$.
 
 We finish by proving that $(\ref{Prp:ExchangeComp1})$ is false in general, when $r\geq 3$. This follows from the fact that if $\Sigma$ is a general finite subset of $2r$ distinct points of $\p^1_{\r}$, the group $\{\alpha\in\aut(\p^1_{\r})\ |\ \alpha(\Sigma)=\Sigma\}$ is trivial. Supposing this fact true, we obtain the result by applying it to the $2r$ boundary points of $I(X,\pi)$. Let us prove the fact.
 The set of $2r$-tuples of $\p^1_{\r}$ is an open subset $W$ of $(\p^1_{\r})^{2r}$. For any non-trivial permutation $\upsilon \in \Sym_{2r}$, we denote by $W_{\upsilon}\subset W$ the set of points $a=(a_{1},\dots,a_{2r})\in W$ such that there exists $\alpha\in \aut(\p^1_{\r})$ with $\alpha(a_{i})=a_{\upsilon(i)}$ for each $i$. Let $a\in W_{\upsilon}$, and take two $4$-tuples  $\Sigma_{1},\Sigma_{2}$ of $a_{i}$'s with $\Sigma_{1}\not=\Sigma_{2}$ and $\Sigma_{2}=\upsilon(\Sigma_{1})$ (this is possible since $\upsilon$ is non-trivial). Then, the cross-ratio of the $a_{i}$'s in $\Sigma_{1}$ and in $\Sigma_{2}$ are the same. This implies a non-trivial condition on $W$. Consequently, $W_{\upsilon}$ is contained in a closed subset of $W$. Doing this for all non-trivial permutations $\upsilon$, we obtain the result.
\end{proof}

\section{Real algebraic models}\label{Sec:RealAlgModels}

The aim of this section is to go further with non-minimal surfaces with $2$ or $3$ connected components. We begin to show how \emph{to separate} infinitely near points to the effect that any such a surface $Y(\r)$ is isomorphic to a blow-up $B_{a_1,\dots,a_m}X(\r)$ where $X$ is minimal and $a_1,\dots,a_m$ are distinct proper points of $X(\r)$. Then, we replace $X(\r)$ by an isomorphic del Pezzo model (Corollary~\ref{Cor:23comp2conicbundles}) and we use the fact that $\aut\bigl(X(\r)\bigr)$ is very transitive on each connected component for such an $X$ (Proposition~\ref{Prp:TransivitySurface})  to prove that in many cases, if two birational surfaces $Y$ and $Z$ have homeomorphic real parts then $Y(\r)$ and $Z(\r)$ are isomorphic. As a corollary, we get that in any cases, $\aut\bigl(Y(\r)\bigr)$ is very transitive on each connected component.
 
\begin{prop}\label{Prp:SeparationPts}
Let $X$ be a  minimal geometrically rational real surface, with $\#X(\r)=2$ or $\#X(\r)=3$, and let $\eta\colon Y\to X$ be a birational morphism.

Then there exists a blow-up $\eta'\colon Y'\to X$,  whose centre is a finite number of distinct real proper points of $X$, and such that $Y'(\r)$ is isomorphic to $Y(\r)$. 

Moreover, we can assume that the isomorphism $Y(\r)\to Y'(\r)$ induces an homeomorphism $\eta^{-1}(M) \to (\eta')^{-1}(M)$ for each connected component $M$ of $X(\r)$.
\end{prop}

\begin{proof}
According to Corollary~\ref{Cor:23comp2conicbundles}, we may assume that $X$ admits two minimal conic bundles $\pi_1\colon X\to \p^1_\r$ and $\pi_2\colon X\to \p^1_\r$ inducing distinct foliations on $X(\c)$. 
Preserving the isomorphism class of $Y(\r)$, we may assume that the points in the centre of $\eta$ are all real (such a point may be a proper point of $X(\r)$ or an infinitely near point). Let us denote by $m$  ($= K_{X}^2 - K_Y^2$) the number of those points. We prove the result by induction on $m$.

The cases $m=0$ and $m=1$ being obvious (take $\eta'=\eta$), we assume that $m\ge 2$. We decompose $\eta$ as $\eta=\theta\circ \epsilon$, where $\epsilon\colon Y\to Z$ is the blow-up of one real point $q\in Z$, and $\theta\colon Z\to Y$ is the blow-up of $m-1$ real points. By induction hypothesis, we may assume that $\theta$ is the blow-up of $m-1$ proper points of $X$, namely $a_{1},\cdots,a_{m-1}\in X(\r)$. 
Moreover,  applying   Proposition~\ref{Prp:TransivitySurface}, we may move the points by an element of $\aut\bigl(X(\r)\bigr)$, and assume that $\pi_{1}(a_{i})\ne \pi_{1}(a_{j})$ and $\pi_{2}(a_{i})\ne \pi_{2}(a_{j})$ for $i\ne j$, and that the fibre of $\pi_{1}$ passing through $a_{i}$ and the fibre of $\pi_{2}$ passing through $a_{i}$ are nonsingular and transverse at $a_{i}$, for each $i$.

If $\theta(q)\notin\{a_{1},\dots,a_{m-1}\}$, then $\eta$ is the blow-up of $m$ distinct proper points of $X$, hence we are done. 
Otherwise, assume that $\theta(q)=a_{1}$. 
We write $E=\theta^{-1}(a_{1})\subset Z$, and denote by $F_{i}\subset Z$ the strict pull-back by $\eta$ of the fibre of $\pi_{i}$ passing through $a_{1}$, for $i=1,2$. Then, $F_{1}$ and $F_{2}$ are two $(-1)$-curves which do not intersect. 
Hence, the point $q\in E$ belongs to at most one of the two curves, so we may assume that $q\notin F_{1}$. 
Denote by $\theta_{2}\colon Z\to X_{2}$ the contraction of the $m-1$ disjoint $(-1)$-curves $F_{1},\theta^{-1}(a_{2}),\dots,\theta^{-1}(a_{m-1})$. 
Since $q$ does not belong to any of these curves, $\eta_{2}=\theta_{2}\circ \epsilon$ is the blow-up of $m-1$ distinct proper points of $X_{2}$. It remains to find an isomorphism $\gamma\colon X_{2}(\r)\to X(\r)$ such that for each connected component $M$ of $X(\r)$, $\gamma\eta_{2}$ sends $\eta^{-1}(M)$ on $M$.

Denoting $\pi'=\pi_{1}\circ\theta\circ \theta_{2}^{-1}$, the map $\psi=\theta_{2}\circ \theta^{-1}$ is a birational map of conic bundles $(X,\pi_{1})\dasharrow (X_{2},\pi')$, which factorizes as the blow-up of $a_{1}$, followed by the contraction of the strict transform of the fibre passing through $a_{1}$. Therefore, the conic bundle $(X_{2},\pi')$ is minimal. Since $\#X(\r)>1$ and $\pi'\psi=\pi_{1}$, Proposition~\ref{Prp:Equivalences} yields the existence of an isomorphism $\gamma\colon X_{2}(\r)\to X(\r)$ such that $\pi_{1}\gamma=\pi'$. 
Observe that 
$\gamma\eta_{2}\circ \eta^{-1}=\gamma\theta_{2}\circ \theta^{-1}=\gamma\psi$ is a birational map $X\dasharrow X$ which satisfies $\pi\circ(\gamma\eta_{2}\circ \eta^{-1})=\pi$. 
Consequently, for any connected component $M$ of $X(\r)$, which corresponds to $\pi^{-1}(V) \cap X(\r)$, for some interval $V\subset \p^1_{{\r}}$, we find $\pi(\gamma\eta_{2}\eta^{-1}(M))=\pi(M)=V$, thus $\gamma\eta_{2}$ sends $\eta^{-1}(M)$ on $M$.
\end{proof}

\begin{cor}
Let $X$ be a minimal geometrically rational real surface, such that $\#X(\r)=2$ or $\#X(\r)=3$, and let $\eta\colon Y\to X$, $\epsilon\colon Z\to X$ be two birational morphisms. Denote by $M_{1},\dots,M_{r}$ the connected components of $X(\r)$ $(r=2,3)$. Then, the following are equivalent:
\begin{enumerate}
\item
$\eta^{-1}(M_{i})\subset Y(\r)$ and $\epsilon^{-1}(M_{i})\subset Z(\r)$ are homeomorphic for each $i$;
\item
there exists an isomorphism $Y(\r)\to Z(\r)$ which  induces an homeomorphism $\eta^{-1}(M_{i})\to \epsilon^{-1}(M_{i})$ for each $i$.
\end{enumerate}
\end{cor}
\begin{proof}
$(2)\Rightarrow (1)$ being obvious, let us prove the converse.
According to Proposition~\ref{Prp:SeparationPts}, we may assume that $\eta$ and $\epsilon$ are the blow-ups of a finite number of distinct real proper points of $X$. Denote by $\Sigma_{\eta}$ and $\Sigma_{\epsilon}$ these two finite sets. For each $i$, the fact that $\eta^{-1}(M_{i})\subset Y(\r)$ and $\epsilon^{-1}(M_{i})\subset Z(\r)$ are homeomorphic implies that the numbers of points of $\Sigma_{\eta}\cap M_{i}$ and $\Sigma_{\epsilon}\cap M_{i}$ coincide.

By Corollary~\ref{Cor:23comp2conicbundles} and Proposition~\ref{Prp:TransivitySurface}, $\aut\bigl(X(\r)\bigr)$ is very transitive on each connected component of $X(\r)$. In particular, there exists an element $\alpha\in\aut\bigl(X(\r)\bigr)$ such that $\alpha(M_{i}) = M_{i}$ for each $i$ and $\alpha(\Sigma_{\eta}) = \Sigma_{\epsilon}$. Then, $\psi=\epsilon^{-1}\alpha\eta\colon Y(\r)\to Z(\r)$ is the wanted isomorphism.
\end{proof}

\begin{cor}\label{Cor:TransitivByComponent23}
Let $Y$ be a geometrically rational real surface with\, $\#Y(\r) = 2$ or $\#Y(\r) = 3$. 
Let $(p_1,\dots,p_n)$ and $(q_1,\dots,q_n)$ be two $n$-tuples of distinct points of $Y(\r)$ such that $p_i$ and $q_i$ belong to the same connected component for each $i$. 

Then, there exists an element $\alpha \in \aut\bigl(Y(\r)\bigr)$, which leaves each connected component of $Y(\r)$ invariant and such that $\alpha(p_i) = q_i$ for each~$i$.
\end{cor}

\begin{proof}Let $\eta\colon Y\to X$ be a birational morphism to a minimal real surface $X$; observe that $\#X(\r)=\#Y(\r)$. 
According to Corollary~\ref{Cor:23comp2conicbundles}, we may assume that $X$ admits two minimal conic bundles $\pi_1\colon X\to \p^1_\r$ and $\pi_2\colon X\to \p^1_\r$ inducing distinct foliations on $X(\c)$. By Proposition~\ref{Prp:SeparationPts}, we can suppose that $\eta$ is the blow-up of $m$ distinct real proper points $a_{1},\dots,a_{m}\in X$. We prove the result by induction on $m$.

If $m=0$, which means that $X = Y$, the result  follows from Proposition~\ref{Prp:TransivitySurface}. 

If $m>0$, denote by  $\eta_0\colon Z\to X$ the blow-up of $a_{1},\dots,a_{m-1}$ ($\eta_{0}$ is the identity if $m=1$), and by $\eta_{1}\colon Y\to Z$ the blow-up of $b = \eta_{0}^{-1}(a_{r})$.

Applying Proposition~\ref{Prp:TransivitySurface},
we may assume that $\pi_1(a_i)\not=\pi_1(a_j)$ and $\pi_2(a_i)\ne \pi_2(a_j)$ for $i\ne j$, and that the fibre of $\pi_{1}$ passing through $a_{i}$ and the fibre of $\pi_{2}$ passing through $a_{i}$ are nonsingular and transverse at $a_{i}$, for each $i$.  
Let us denote by $E\subset Y$ the exceptional curve $\eta_1^{-1}(b)$ of $\eta_1$ and by $F_{i}$ the strict transform on $Y$ of the fibre of $\pi_i$ passing through $a_m$, for $i=1,2$. 
Then $E$, $F_{1}$ and $F_{2}$ are three $(-1)$-curves, $F_{1}$ and $F_{2}$ do not intersect, and $E$ intersect transversally each of the $F_{i}$. 
By induction hypothesis, we may use the lift of an element of $\aut\bigl(Z(\r)\bigr)$ which fixes $b$ to assume that no one of the points $p_i$ belongs to $F_{1}\backslash E$, $F_{2}\backslash E$ or to $\eta^{-1}(a_i)$ for $i=1,\dots,m-1$. 
Then the group $G=\{\alpha \in \aut\bigl(X(\r)\bigr)\ |\ \pi_{1}\alpha=\pi_{1}, \alpha\mbox{ fixes }a_1,\dots,a_m,\eta(p_1),\dots,\eta(p_n)\}$, acts transitively on $E\backslash F_{1}$  (Lemma~\ref{Lem:TransivityTangentConicBundle}). 
Lifting a well-chosen element of this group in $\aut\bigl(Y(\r)\bigr)$, we may move the points $p_i$ and assume that no one of the $p_i$  belongs to $F_{2}$ (i.e.\ we can avoid $F_{2}\cap E$). 
Denote by $\eta'\colon Y\to X'$  the contraction of the disjoint $(-1)$-curves $F_2,\eta^{-1}(a_1),\dots
\eta^{-1}(a_{m-1})$. 

Then, the birational map $\psi=\eta'\eta^{-1}\colon X\dasharrow X'$ is a birational map of conic bundles $(X,\pi_2)\dasharrow (X',\pi')$, where $\pi'=\pi_2\psi^{-1}$, which consists of the blow-up of $a_{m}$, followed by the contraction of the strict transform of the fibre passing through $a_{m}$. 
Therefore, the conic bundle $(X',\pi')$ is minimal. Since $\#X(\r)>1$,
Proposition~\ref{Prp:Equivalences} yields the existence of an isomorphism $\gamma\colon X'(\r)\to X(\r)$ such that $\pi_{2}\gamma=\pi'$.
Therefore, there exists an element $\beta\in\aut \bigl(X'(\r)\bigr)$ which fixes all the points blown-up by $\eta'$, which fixes all the points $\{\eta'(p_i),p_i\notin E\}$, and which sends the points $\{\eta'(p_i),p_i\in E\}$ outside of $\eta'(E)$. Applying the lift of $\beta$ on $\aut\bigl(Y(\r)\bigr)$, we may assume that none of the points $p_i$ belongs to $E$. Doing the same manipulation with the $q_i$, it remains to use the lift of an element of $\aut\bigl(Z(\r)\bigr)$ which fixes $b$ and sends $\eta_1(p_i)$ on $\eta_1(q_i)$ for each $i$.
\end{proof}

\section{Proof of the main results}\label{Sec:Proofs}

The proof of Theorem~$\ref{Thm:Diff+Bir=DiffBir}$ was given at the end of Section~\ref{Sec:ConicBunDPS}. Now, we deduce the others results stated in the introduction from the results of Sections~\ref{Sec:VeryTransitive} and \ref{Sec:RealAlgModels}. The following lemma serves to prove most of them.

\begin{lem}\label{Lem:Magic}
Let $(X,\pi)$ be a minimal real conic bundle, such that $I(X,\pi)$ is the union of $r$ intervals $I_{1},\dots,I_{r}$, with $r=2$ or $r=3$.

Let $\eta_{Y}\colon Y\to X$ and $\eta_{Z}\colon Z\to X$ be two birational morphisms. For $i = 1,\dots,r$, we write ${X}_{i}=\pi^{-1}(I_{i})\cap X(\r)$, $Y_{i}=\eta_{Y}^{-1}(X_{i})\cap Y(\r)$ and $Z_{i}=\eta_{Z}^{-1}(X_{i})\cap Z(\r)$.

Let $p_{1},\dots,p_{n}\in Y(\r)$, $q_{1},\dots,q_{n}\in Z(\r)$ be two $n$-tuples of distinct points, and assume the existence of an homeomorphism $h\colon Y(\r)\to Z(\r)$ which sends $p_{i}$ on $q_{i}$ for each $i$, and sends $Y_{i}$ on $Z_{\nu(i)}$, where $\nu\in \Sym_{r}$ is a permutation of $\{1,\dots,r\}$.
Then, the following are equivalent:
\begin{enumerate}
\item\label{Lem:Magic1}
There exists an isomorphism $\beta\colon Y(\r)\to Z(\r)$ which sends  $Y_{i}$ on $Z_{\nu(i)}$ for each $i\in \{1,\dots,r\}$ and sends $p_{j}$ on $q_{j}$ for each $j\in \{1,\dots,n\}$.
\item\label{Lem:Magic2}
There exists an automorphism $\alpha\in \aut(\p^1_{\r})$ which sends $I_{i}$ on $I_{\nu(i)}$ for each $i\in \{1,\dots,r\}$.
\end{enumerate}

Moreover, both assertions are true if $r = 2$, and false in general when $r=3$.
\end{lem}

\begin{proof}
Observe that the $X_{i}$ (respectively the $Y_{i}$, $Z_{i}$) are the connected components of $X(\r)$ (respectively of $Y(\r)$, $Z(\r)$).

[$\ref{Lem:Magic1})\Rightarrow \ref{Lem:Magic2})$] The map $\eta_{Z}\beta\eta_{Y}^{-1}$ is a birational self-map of $X$, which restricts to an isomorphism $\varphi\colon V\to W$, where $V$ and $W$ are two real Zariski open subsets of $X$. Moreover, the hypothesis on $\beta$ implies that $\varphi(V(\r)\cap X_{i})=W(\r)\cap X_{\nu(i)}$. The existence of $\alpha$ is provided by Proposition~\ref{Prp:ExchangeComp}.

[$\ref{Lem:Magic2})\Rightarrow \ref{Lem:Magic1})$] Proposition~\ref{Prp:ExchangeComp} yields the existence of $\gamma\in \aut(X(\r),\pi)$ such that $\gamma(X_{i})=X_{\nu(i)}$. We may thus assume that $\nu$ is the identity. 
According to Proposition~\ref{Prp:SeparationPts}, we may moreover suppose that $\eta_{Y}$ and $\eta_{Z}$ are the blow-ups of a finite set of disjoint real proper points of $X$. Since $Y_{i}$ is homeomorphic to $Z_i$ for each $i$,  $\eta_{Y}$ is the blow-up of $a_{1},\dots,a_{m}$ and $\eta_{Z}$ is the blow-up of $b_{1},\dots,b_{m}$, where $a_{j}$ and $b_{j}$ belong to the same connected component of $X(\r)$ for each $j$.
Then, there exists an element of $\aut\bigl(X(\r)\bigr)$ which preserves each connected component of $X$ and sends $a_{j}$ on $b_{j}$ for each~$j$ (Corollary~\ref{Cor:TransitivByComponent23}). We may thus assume that $Y=Z$, and conclude by applying  Corollary~\ref{Cor:TransitivByComponent23} to~$Y$.

The fact that $\ref{Lem:Magic2})$ is true when $r=2$ and false in general when $r=3$ was proved in Proposition~\ref{Prp:ExchangeComp}.
\end{proof}

The following case shares many features with the rational case.

\begin{thm}\label{Thm:2CompVeryTransitive}
Let $X$ be a nonsingular geometrically rational real projective surface, and assume that $\#X(\r)=2$. Then the action of the group $\aut\bigl(X(\r)\bigr)$ on  $X(\r)$ is very transitive.
\end{thm}

\begin{proof}
Let $Y$ be a nonsingular geometrically rational real projective surface, with $\#Y(\r)=2$. Let $(p_{1},\dots,p_{n})$ and $(q_{1},\dots,q_{n})$ be two $n$-tuples of points which are compatible. We want to prove the existence of $\alpha\in \aut\bigl(Y(\r)\bigr)$ such that $\alpha(p_{i})=q_{i}$ for each $i$.

If $p_{i}$ and $q_{i}$ are in the same connected component of $Y(\r)$, the result follows from Corollary~\ref{Cor:TransitivByComponent23}.

Otherwise, the compatibility means that the two components of $X(\r)$ are homeomorphic and that $p_{i}$ and $q_{i}$ are in a distinct component for each $i$. Lemma~\ref{Lem:Magic} provides the existence of an element of $\aut\bigl(Y(\r)\bigr)$ which permutes the two connected components of $Y(\r)$. This reduces the situation to the previous case.
\end{proof}

Before proving Theorems~\ref{Thm:3compTrans} and \ref{thm:verytrans}, we describe the cases where the group of automorphisms is not very transitive.
\begin{lem}\label{Lem:NonTransitivity}Let $X$ be a nonsingular real projective surface, and assume that either $X$  is \emph{not} geometrically rational or $\#X(\r)>3$. The group $\aut\bigl(X(\r)\bigr)$ is then not very transitive on each connected component, and is neither $2$-transitive.
\end{lem}
\begin{proof}
If $X$ has Kodaira dimension $2$, (surface of general type), it has only finitely many birational self-maps
(see e.g. \cite{ueno75}.)
If $X$ has Kodaira dimension $1$, every birational self-map of $X$ preserves the elliptic fibration induced by $|K_{X}|$.
If $X$ has Kodaira dimension $0$, and $X$ is minimal, then $\bir(X)=\aut(X)$. The group $\aut(X)$ is an algebraic group of dimension $1$ or $2$ (its neutral component is an elliptic curve or an Abelian surface). Thus, $\bir(X)$ can not be $2$-transitive. The case when $X$ is not minimal is deduced from this case.

If $X$ is a surface with Kodaira dimension $-\infty$, then $X$ is uniruled. If furthermore, $X$ is not geometrically rational and $X(\r)$ is non-empty, then the Albanese map $X \to C$ is a real ruling over a curve with genus $g(C) > 0$, see e.g. \cite[V.(1.8)]{Sil}, and the Albanese map is preserved by any birational self-map.

The remaining case is when $X$ is geometrically rational and $\#X(\r) > 3$; we prove now that the group $\aut\bigl(X(\r)\bigr)$ is not transitive. Denote by $\eta\colon X\to X_{0}$ a birational morphism to a minimal real surface, and observe that $\#X_{0}(\r) = \#X(\r) > 3$.   
Let us discuss the two cases for $X_{0}$ given by Theorem~\ref{Thm:ClassicMinimal}.
If $X_{0}$ is a del Pezzo surface with $\rho(X_{0})=1$, then $\aut\bigl(X(\r)\bigr)$ is countable (Corollary~\ref{Cor:DPcountable}), thus  $\aut\bigl(X(\r)\bigr)$ cannot be transitive.
The other case is when $\rho(X_{0})=2$. Then, $X_{0}$ endows a real conic bundle structure $(X_{0},\pi_{0})$, and $\bir(X_{0})=\bir(X_{0},\pi_{0})$ (Theorem~ \ref{Thm:Birapport}). 
Since the action of $\bir(X_{0},\pi_{0})$ on the basis of the conic bundle is finite (there are too much boundary points),  neither $\aut\bigl(X_{0}(\r)\bigr)$ nor $\aut\bigl(X(\r)\bigr)$ may be transitive.
\end{proof}

\begin{proof}[Proof of Theorem~$\ref{Thm:3compTrans}$] When $X$ is not geometrically rational or $\#X(\r)>3$, $\aut(X(\r))$ is not very transitive on connected components by Lemma~\ref{Lem:NonTransitivity}. In the remaining cases, $\aut(X(\r))$ is very transitive on connected components. When $\#X(\r)=2,3$, this is Corollary~\ref{Cor:TransitivByComponent23}. When $\#X(\r)=1$, this is the main result of \cite{hm3}.\end{proof}

\begin{proof}[Proof of Theorem~$\ref{thm:verytrans}$]According to Lemma~\ref{Lem:NonTransitivity}, we can
assume from now on that $X$ is a geometrically rational surface with $\#X(\r)\le 3$.
When $\#X(\r)=1$, $X$ is rational; the fact that $\aut\bigl(X(\r)\bigr)$ is $n$-transitive for every $n$ (and thus very transitive) is the main result of \cite{hm3}. When $\#X(\r)=2$, $\aut\bigl(X(\r)\bigr)$ is very transitive by Theorem~\ref{Thm:2CompVeryTransitive}. 

When $\#X(\r) = 3$, $\aut\bigl(X(\r)\bigr)$ is very transitive on each connected component (Theorem~\ref{Thm:3compTrans}). Thus, $\aut\bigl(X(\r)\bigr)$ is very transitive if and only if for any homeomorphism $h\colon X(\r)\to X(\r)$, there exists $\beta\in \aut\bigl(X(\r)\bigr)$ which permutes the components of $X(\r)$ in the same way that $h$ does. When these conditions are not satisfied, $\aut\bigl(X(\r)\bigr)$ is not $2$-transitive.

Let $X(\r)=M_1\sqcup M_2 \sqcup M_3$ be the decomposition into connected components.
If there is no pair $(i,j)$ such that $M_i\sim M_j$, then there is no nontrivial such $h$, hence $\aut\bigl(X(\r)\bigr)$ is very transitive.
If $M_1\sim M_2 \not\sim M_3$ or $M_1\sim M_2 \sim M_3$, the possibilities when this occur follow from Lemma~\ref{Lem:Magic}.  

For example, when $X$ is minimal (therefore $M_1\sim M_2 \sim M_3\sim S^2$), it admits a minimal real conic bundle structure $(X,\pi)$ (Theorem~\ref{Thm:ClassicMinimal} and Proposition~\ref{Prp:TopMinimal}), where $\pi$ has $6$ singular fibres. 
Then, $\aut\bigl(X(\r)\bigr)$ is very transitive if and only if $\{\alpha\in \aut(\p^1_{\r})\ |\ \alpha(I(X,\pi)) = I(X,\pi)\}$ acts transitively on the three intervals of $I(X,\pi)$. This is true in some special cases, but false in general. When $X$ is not minimal, $\aut\bigl(X(\r)\bigr)$ is very transitive for example when there is no pair of homeomorphic connected components of $X(\r)$, or when $X$ is the blow-up of a minimal surface $Y$ with a very transitive group $\aut\bigl(Y(\r)\bigr)$.
\end{proof}

\begin{proof}[Proof of Theorem~$\ref{thm:unicity}$]Let $X,Y$ be two geometrically rational real surfaces, and assume that $\#
X(\r)\leq 2$. We assume that $X$ is birational to $Y$ and that $X(\r)$ is homeomorphic to $Y(\r)$, and prove that $X(\r)$
is isomorphic to $Y(\r)$.

Remark that all geometrically rational surfaces with connected
real part are birational to each others, thus in this case the
statement follows from the unicity of rational models \cite{bh}. We may thus assume that $\#X(\r)=2$. Denote by $\eta_{X}\colon X\to X_{0}$ and $\eta_{Y}\colon Y\to Y_{0}$ birational morphisms to minimal real surfaces.

Since $X_{0}$ and $Y_{0}$ are birational, $X_{0}(\r)$ and $Y_{0}(\r)$ are isomorphic (Theorem~\ref{Thm:Diff+Bir=DiffBir}), so we may assume that $X_{0}=Y_{0}$. The result now follows from Lemma~\ref{Lem:Magic}.
\end{proof}

\begin{proof}[Proof of Corollary~$\ref{Cor:Models}$]
If $M$ is connected, and $M$ is non-orientable or $M$ is orientable with genus $g(M) \leq 1$, then it admits a unique geometrically rational model by \cite[Corollary~8.1]{bh}. Moreover, this model is in fact rational.

Conversely let $M$ be a compact $\mathcal{C}^\infty$-surface and assume that $M$ admits a unique geometrically rational model $X$. 
The existence of such a model implies, by Comessatti's theorem \cite{Com2}, that any connected component of $M$ is non-orientable or is orientable with genus $g \leq 1$.
The unicity means that for any geometrically rational model $Y$ of $M$, then $Y(\r)$ is isomorphic to $X(\r)$.
In particular, this implies that all geometrically rational models of $M$ belong to a unique birational class. 
From Theorem~\ref{Thm:Birapport} and Proposition~\ref{Prop:DP1DP2rho1}, this means that $X$ is rational. 
It remains to observe that when $X$ is rational, $X(\r)$ is connected, and is either non-orientable or orientable of genus $\le 1$. When $X$ is minimal, this follows from Proposition~\ref{Prp:TopMinimal}. 
Then, blowing-up points on a surface either does nothing on the topology of the real part (if the points blown-up are imaginary), or it gives a non-orientable real part (if the points blown-up ar real).
\end{proof}
We finish by a result on non-density. In \cite{km1}, it is proved that $\aut\bigl(X(\r)\bigr)$ is dense in $\diff\bigl(X(\r)\bigr)$ when $X$ is a geometrically rational surface with $\#X(\r)=1$ (or equivalently when $X$ is rational). 
In the cited paper, it is said that $\#X(\r)=2$ is probably the only other case where the density holds. 
The following collect the known results in this direction. The first two of them are new.

\begin{prop}\label{Prp:Density}
Let $X$ be a geometrically rational surface. 
\begin{itemize}
\item If $\#X(\r)\geq 5$, then $\aut\bigl(X(\r)\bigr)$ is not dense in $\diff\bigl(X(\r)\bigr)$;
\item if $\#X(\r)=4$, and either $X$ is the blow-up of a minimal conic bundle or $\rho(X)=1$, then $\aut\bigl(X(\r)\bigr)$ is not dense in $\diff\bigl(X(\r)\bigr)$;
\item if $\#X(\r)=3$ and $X$ is minimal, then $\aut\bigl(X(\r)\bigr)$ is not dense in $\diff\bigl(X(\r)\bigr)$ for a general $X$;
\item if $\#X(\r) =1$, then $\aut\bigl(X(\r)\bigr)$ is dense in $\diff\bigl(X(\r)\bigr)$. 
\end{itemize}

\end{prop}

\begin{proof} The case $\#X(\r) =1$ is the main result of \cite{km1}. Assume from now on that $\#X(\r) \geq 3$, and
denote by $\eta\colon X\to X_{0}$ a birational morphism to a minimal real surface, and observe that $\#X_{0}(\r) = \#X(\r) \geq 3$.   
Let us discuss the two cases for $X_{0}$ given by Theorem~\ref{Thm:ClassicMinimal}.

Assume that $X_{0}$ is a del Pezzo surface with $\rho(X_{0})=1$. If the degree of $X_{0}$ is $1$ then $\bir(X_{0})$ is finite (Corollary~\ref{Cor:DPcountable}), thus $\aut\bigl(X(\r)\bigr)$ cannot be dense. 
If $X_{0}$ has degree $2$, then $\#X_{0}(\r)=4$ (Proposition~\ref{Prp:TopMinimal}), so $\#X(\r)=4$ too. Since $\aut\bigl(X_{0}(\r)\bigr)=\aut(X_{0})$ is finite, $\aut\bigl(X_{0}(\r)\bigr)$ cannot be dense (but maybe $\aut\bigl(X(\r)\bigr)$ could be).

The other case is when $\rho(X_{0})=2$. Then, $X_{0}$ endows a real conic bundle structure $(X_{0},\pi_{0})$. If $\#X(\r)=\#X_{0}(\r)\geq 4$, then $\bir(X_{0})=\bir(X_{0},\pi_{0})$ (Theorem~ \ref{Thm:Birapport}), so $\aut\bigl(X(\r)\bigr)$ is not dense. 
If $\#X_{0}(\r)=3$, then in general $\aut\bigl(X_{0}(\r)\bigr)$ does not exchanges the connected component of $X_{0}(\r)$. 
Consequently, $\aut\bigl(X_{0}(\r)\bigr)$ is not dense (but maybe $\aut\bigl(X(\r)\bigr)$ could be, if the connected components of $X(\r)$ are not homeomorphic).
\end{proof}


\begin{thebibliography}{Mang06}
\bibitem[BH07]{bh} 
I.~Biswas, J.~Huisman, 
\textit{Rational real algebraic models of topological surfaces}, Doc.
   Math.~\textbf{12} (2007), 549--567.

\bibitem[Bla09a]{BlaGGD}
J. Blanc, {\it Linearisation of finite Abelian subgroups of the Cremona group of the plane.}\\
Groups Geom. Dyn. {\bf 3} (2009), no. 2, 215-266.

\bibitem[Bla09b]{BlaTG}
J. Blanc, {\it Sous-groupes alg\'ebriques du groupe de Cremona.} Transform. Groups {\bf  14} (2009), no. 2, 249-285.

\bibitem[BCR98]{bcr} J.~Bochnak, M.~Coste, M.-F.~Roy,
 Real algebraic geometry, Ergeb. Math. Grenzgeb.  (3),
vol.~36, Springer Verlag, 1998.

\bibitem[Com12]{Com1}
A. Comessatti, {\it Fondamenti per la geometria sopra superfizie razionali dal punto di vista reale,} Math. Ann. {\bf 73} (1912) 1-72.

\bibitem[Com14]{Com2}
A. Comessatti, {\it Sulla connessione delle superfizie razionali reali}, Annali di Math. {\bf 23}(3) (1914) 215-283.

\bibitem[Dem76]{bib:DemDPezzo}
M. Demazure, {\it Surfaces de Del Pezzo II}. S\'eminaire sur les singularit\'es des surfaces, Palaiseau, France, (1976-1977), Lecture Notes in Mathematics, 777, 22-70. 

\bibitem[DI06]{DoIs}
I.V. Dolgachev, V.A. Iskovskikh, {\it Finite subgroups of the plane Cremona group.} to appear in "Algebra, Arithmetic and Geometry -- Manin Festschrift" \texttt{arXiv:math/0610595v2}.

\bibitem[Hir64]{hiro} 
H.~Hironaka,
\textit{Resolution of singularities of an algebraic variety over a
              field of characteristic zero. {I}, {II}},
Ann. of Math. (2) \textbf{79} (1964), 109--203; ibid., 205--326.


\bibitem[HM09a]{hm3} 
J.~Huisman, F.~Mangolte,
\textit{The group of automorphisms of a real rational surface is
$n$-transitive}, Bull. London Math. Soc. {\bf 41}, 563--568	(2009).


\bibitem[HM09b]{hm4} 
J.~Huisman, F.~Mangolte,
\textit{Automorphisms of real rational surfaces and weighted blow-up singularities}, manuscripta math. (2010), in press.

\bibitem[HM09c]{hm5} 
J.~Huisman, F.~Mangolte,
\textit{Algebraic models of orientable surfaces}, in preparation.


\bibitem[Isk79]{bib:IskMinimal}
V.A. Iskovskikh, {\it Minimal models of rational surfaces over arbitrary fields}. Izv. Akad. Nauk SSSR Ser. Mat. {\bf 43} (1979), no 1, 19-43, 237.
\bibitem[Isk96]{IskFact}
V.A. Iskovskikh, {\it Factorization of birational mappings of rational surfaces from the point of view of Mori theory.} Uspekhi Mat. Nauk 51 (1996) no 4 (310), 3-72.

\bibitem[Kol97]{Kol}
J. Koll\'ar, 
{\it Real algebraic surfaces}, \texttt{arXiv:alg-geom/9712003v1.}

\bibitem[Kol01]{ko-topo-2000}
J. Koll\'ar,
{\it The topology of real algebraic
varieties}, Current developments in mathematics 2000, 197--231, Int.
Press, Somerville, MA, 2001.

\bibitem[KM09]{km1} J.~Koll\'ar, F.~Mangolte,
\textit{Cremona transformations and diffeomorphisms of surfaces},
Adv. in Math. {\bf 222}, 44-61 (2009).


\bibitem[Mang06]{Ma06} F.~Mangolte, 
\textit{Real algebraic
   morphisms on 2-dimensional conic bundles}, Adv.\
   Geom.~\textbf{6} (2006), 199--213.


\bibitem[Mani67]{bib:Man}
Yu. Manin, {\it Rational surfaces over perfect fields, II.} Math. USSR - Sbornik {\bf 1} (1967), 141-168.

\bibitem[RV05]{rv} 
F.~Ronga,  T.~Vust,
{Diffeomorfismi birazionali del piano proiettivo reale}, 
{Comm. Math. Helv.}~\textbf{80} (2005), 517--540.


\bibitem[Sil89]{Sil}
R. Silhol, {\it Real algebraic surfaces, }Springer Lecture Notes vol. 1392, 
1989.


\bibitem[Tog73]{to}
A.~Tognoli,
\textit{Su una congettura di {N}ash},
Ann. Scuola Norm. Sup. Pisa (3)~\textbf{27}
(1973),
167--185.

\bibitem[Uen75]{ueno75} 
K.~Ueno,
\textit{Classification theory of algebraic varieties and compact complex
   spaces}, 
   Lecture Notes in
Math.~\textbf{439}, 
Springer-Verlag,
Berlin, 1975.

\end{thebibliography}
\end{document}